\documentclass[12pt,reqno]{amsart}

\usepackage{amssymb,amsmath,graphicx,amsfonts,euscript,mathrsfs}
\usepackage{color}

\numberwithin{equation}{section}

\def\H{{\cal H}}

\def\R{\mathbb{R}}
\def\Z{\mathbb{Z}}

\def\H1{H^1(\R)}

\newtheorem{thm}{Theorem}
\newtheorem{lem}{Lemma}

\newtheorem{prop}{Proposition}

\newtheorem{defn}{Definition}

\newtheorem{remark}{Remark}

\makeatletter
\newcommand{\Extend}[5]{\ext@arrow0099{\arrowfill@#1#2#3}{#4}{#5}}
\makeatother

\begin{document}

\setcounter{page}{1}

\title[Solitary wave solutions for DNLS]{Instability of the solitary wave solutions for the genenalized derivative Nonlinear Schr\"odinger equation in the critical frequency case}

\author{Zihua Guo}
\address{School of Mathematical Sciences, Monash University, VIC 3800, Australia}
\email{zihua.guo@monash.edu}
\thanks{}

\author{Cui Ning}
\address{School of Mathematics, South China University of Technology, Guangzhou, Guangdong 510640, P.R.China}
\email{cuiningmath@gmail.com}

\thanks{Z. Guo is supported by ARC DP170101060.}

\author{Yifei Wu}
\address{Center for Applied Mathematics, Tianjin University,
Tianjin 300072, P.R.China}
\email{yerfmath@gmail.com}
\thanks{}

%\subjclass[2010]{Primary  35Q55; Secondary 35A01}

%\date{\today}


\maketitle

\begin{abstract}\noindent
  We study the stability theory of solitary wave solutions for the generalized derivative nonlinear Schr\"odinger equation
  $$
  i\partial_{t}u+\partial_{x}^{2}u+i|u|^{2\sigma}\partial_x u=0.
  $$
  The equation has a two-parameter family of solitary wave solutions of the form
\begin{align*}
\phi_{\omega,c}(x)=\varphi_{\omega,c}(x)\exp{\big\{ i\frac c2 x-\frac{i}{2\sigma+2}\int_{-\infty}^{x}\varphi^{2\sigma}_{\omega,c}(y)dy\big\}}.
\end{align*}
Here $ \varphi_{\omega,c}$ is some real-valued function.
It was proved in \cite{LiSiSu1} that the solitary wave solutions are stable if $-2\sqrt{\omega }<c <2z_0\sqrt{\omega }$, and unstable if $2z_0\sqrt{\omega }<c <2\sqrt{\omega }$ for some $z_0\in(0,1)$. We prove the instability at the borderline case $c =2z_0\sqrt{\omega }$ for $1<\sigma<2$, improving the previous results in \cite{Fu-16-DNLS} where $3/2<\sigma<2$.
\end{abstract}

\section{Introduction}
In this paper, we study the stability theory of the solitary wave solutions for the genenalized derivative nonlinear Schr\"odinger equation:
\begin{equation}\label{eqs:gDNLS}
   \aligned
    &i\partial_{t}u+\partial_{x}^{2}u+i|u|^{2\sigma}\partial_x u=0,\qquad t\in \R,\ x\in \R
    \\
    \endaligned\\
\end{equation}
for $\sigma >0$. It describes an Alfv\'{e}n wave and appears in plasma physics, nonlinear optics, and so on (see \cite{MOMT-PHY, M-PHY}).
In the case of $\sigma=1$, by a suitable gauge transformation, \eqref{eqs:gDNLS} is transformed to the standard derivative nonlinear Schr\"odinger equation:
\begin{align}\label{DNLS0}
i\partial_tu+\partial_x^2u+i\partial_x(|u|^2u)=0.
\end{align}

This equation \eqref{DNLS0} was widely studied.
The local well-posedness was proved by Hayashi and Ozawa \cite{HaOz-92-DNLS, HaOz-94-DNLS} in the energy space $H^1(\R)$ and  by Guo and Tan \cite{GuTa91} in the smooth space.
In the paper of \cite{HaOz-92-DNLS}, the authors proved the global well-posedness in the energy space when the initial data $u_0$ satisfies the mass condition $\|u_0\|_{L^2}<\sqrt{2\pi}$. This condition seems natural for global well-posedness in view of the mass critical nonlinear Schr\"odinger equation and generalized KdV equation, as it ensures a priori estimate of $H^1$-norm from mass and energy conservations. However, recently, the third author extended the condition to $\|u_0\|_{L^2}<2\sqrt{\pi}$ in  \cite{Wu1,Wu2}, in which the key ingredient in the proof is the use of the
momentum conservation. A simplified proof was later given by the first and third authors in their paper \cite{Guo-Wu-15-DNLS}, where the global well-posedness in $H^\frac12(\R)$ was also proved under the same mass constraint. The problems for large mass are still unclear at the moment. In \cite{FuHaIn-16-DNLS}, Fukaya, Hayashi and Inui constructed a class of large global solution with high oscillation. In the papers of Cher, Simpson and Sulem \cite{ChSiSu-DNLS}, Jenkins, Liu, Perry, Sulem \cite{JeLiPeSu-1, LiPeSu,LiPeSu1,LiPeSu2}, Pelinovsky and Shimabukuro \cite{PeSh-DNLS,PeSh-DNLS-2}, the authors constructed a class of global solution by using the inverse scattering method. On the long-time behavior and modified scattering
theory, see \cite{GuHaLinNa-DNLS} and references therein.
On the low regularity theory, see \cite{BiLi-01-Illposed-DNLS-BO,CKSTT-01-DNLS, CKSTT-02-DNLS,Gr-05, Grhe-95, GuReWa-DNLS,Herr,Mosincat,MoOh, Miao-Wu-Xu:2011:DNLS, Ta-99-DNLS-LWP, Ta-16-DNLS-LWP} and the reference therein.

In the case of $\sigma\ne 1$, the Cauchy problems of \eqref{eqs:gDNLS} have been investigated by many researchers. In the case of $\sigma>1$, local well-posedness in energy spaces $H^1(\R)$ was studied by Hayashi and Ozawa \cite{HaOz-16-DNLS} for any $\sigma>1$,  by Hao \cite{Hao} in $H^\frac12(\R)$ for any $\sigma>\frac52$, and by Santos \cite{santos}  in $H^\frac12(\R)$ for any $\sigma>1$ and small data. In the case of $\frac12 \le  \sigma<1 $, local well-posedness in energy spaces $H^2(\R)$ was studied by Hayashi and Ozawa \cite{HaOz-16-DNLS}, see also Santos \cite{santos} in the weight Sobolev spaces. In the case of $0< \sigma<\frac12$, local well-posedness in the some weighted spaces was studied by Linares, Ponce and Santos \cite{LiPoSa-Local-DNLS}. Note that in this case, the nonlinear term is not regular enough,  appropriate construction of the working space is needed to handle nonlinearity.  Global well-posedness was studied in \cite{FuHaIn-16-DNLS, HaOz-16-DNLS, Miao-Tang-Xu:2017:DNLS}. In particular, in the case of $0<\sigma<1$ the global existence (without uniqueness) of the solution in $H^1(\R)$ was shown by Hayashi and Ozawa \cite{HaOz-16-DNLS}; while in the case of $\sigma>1$, the global well-posedness of the solution in $H^1(\R)$ was shown by Hayashi and Ozawa \cite{FuHaIn-16-DNLS, Miao-Tang-Xu:2017:DNLS} with some suitable size restriction on the initial datum.

Also, the stability theory was widely studied.  The equation \eqref{eqs:gDNLS} has a two-parameter family of solitary waves,
$$
u_{\omega,c}(t)=e^{i\omega t}\phi_{\omega, c}(x-ct),
$$
where $\phi_{\omega, c}$ is the solution of
\begin{align}
\phi_{\omega,c}(x)=\varphi_{\omega,c}(x)\exp{\big\{\frac c2 i x-\frac{i}{2\sigma+2}\int_{-\infty}^{x}\varphi^{2\sigma}_{\omega,c}(y)dy\big\}},\label{phi}
\end{align}
and
\begin{align*}
\varphi_{\omega,c}(x)=\Big\{\frac{(\sigma+1)(4\omega-c^2)}{2\sqrt\omega \cosh(\sigma\sqrt{4\omega-c^2}\,x)-c}\Big\}^{\frac{1}{2\sigma}}.
\end{align*}
Note that $\phi_{\omega,c}$ is the solution of
\begin{align}\label{Elliptic-comp}
-\partial_x^2\phi+\omega \phi+c i\partial_x\phi-i| \phi|^{2\sigma}\partial_x\phi=0.
\end{align}

When $\sigma= 1$, Colin and Ohta \cite{CoOh-06-DNLS} proved the stability of the soliton waves when $c ^2<4\omega $, see also Guo and Wu \cite{GuWu95} for previous result in the case of $c >0$. The endpoint case $c ^2=4\omega , c >0$ was studied in \cite{Soonsik-W-2014}. Further,
Le Coz and Wu \cite{Stefan-W-15-MultiSoliton-DNLS}, Miao, Tang and Xu \cite{Miao-Tang-Xu:2016:DNLS} proved the stability of the multi-solitary wave solutions. A consequence of these results are a class of the arbitrary large global solutions.
Note that the equation \eqref{phi} can be solved when $4\omega< c^2, c\in\R$ or $4\omega=c^2, c>0$.

In the case of $0<\sigma<1$, Liu, Simpson and Sulem \cite{LiSiSu1} proved that the solitary wave solution $u_{\omega,c}$ is stable for any $-2\sqrt\omega<c<2\sqrt\omega$, Guo \cite{guo-DNLS} further proved the stability of the solitary wave solutions in the endpoint case $0<c=2\sqrt\omega$.  In the case of  $\sigma\ge2$,  the solitary wave solution $u_{\omega,c}$ is unstable for any $-2\sqrt\omega<c<2\sqrt\omega$.

The case $1<\sigma<2$ is more complicated. It was proved by Liu, Simpson and Sulem \cite{LiSiSu1} that there exists $z_0(\sigma)\in (0,1)$, which solves the equation $F_{\sigma}(z)=0$ with
\begin{align*}
F_{\sigma}(z)=(\sigma-1)^2\big[\int_{0}^{\infty}(\cosh y-z)^{-\frac 1\sigma} dy\big]^2-\big[\int_{0}^{\infty}(\cosh y-z)^{-\frac 1\sigma-1}(z\cosh y-1) dy\big]^2,
\end{align*} 
 such that when $-2\sqrt\omega<c<2z_0\sqrt\omega$, the solitary wave solution $u_{\omega,c}$ is unstable and when $2z_0\sqrt\omega<c<2\sqrt\omega$. See also Tang and Xu \cite{TaXu-17-DNLS-Stability} for the stability of the sum of two solitary waves. Further, Fukaya \cite{Fu-16-DNLS} proved that  the solitary waves solution is unstable when $\frac32\le \sigma<2$, $c=2z_0\sqrt\omega$. After the work, the stability theory of the solitary waves solution  when $1< \sigma<\frac32$, $c=2z_0\sqrt\omega$ is the only unknown  case.
  In this paper, we aim to solve this left case.

%Solitary wave solutions:
%\begin{align}
%u_{\omega,c}(t,x)=e^{i\omega t}\phi_{\omega,c}(x-c t)
%\end{align}
%where $(\omega,c)\in \R^2,  c^2<4\omega$,
%\begin{align}
%\phi_{\omega,c}(x)=\varphi_{\omega,c}(x)\exp{i\big\{\frac c2 x-\frac{1}{2\sigma+2}\int_{-\infty}^{x}\varphi^{2\sigma}_{\omega,c}(y)dy\big\}},\label{phi}
%\end{align}

%Know results:
%$\exists z_{0}=z_{0}(\sigma) \in(-1,1)$, such that
%
%Liu, Simpson, Sulem: $c<2z_{0}\sqrt{\omega}$, stable; $c>2z_{0}\sqrt{\omega}$, instable.
%
%Fukaya: $\frac{3}{2}\leq\sigma<2$, $c=2z_{0}\sqrt{\omega}$, instable.
%
%our aim:
%$1<\sigma<2$,  $c=2z_{0}\sqrt{\omega}$, instable.
%For $(\theta,y)\in\R^2 $ and $u\in\H1 $, we define
%$$ T(\theta,y)u=e^{i\theta}u(x-y).$$

Before stating our theorem, we adopt some notations. 
For $\varepsilon>0$, we define
$$ U_\varepsilon(\phi_{\omega,c})
=\{u\in H^1(\mathbb{R}): \inf_{(\theta,y) \in\mathbb{R}^2}\|u-e^{i\theta}\phi_{\omega,c}(\cdot-y)\|_{H^1}<\varepsilon\}.$$

\begin{defn}
We say that the solitary wave solution $u_{\omega,c}$ of \eqref{eqs:gDNLS} is stable
if for any $\varepsilon >0$ there exists $\delta >0$ such that if $u_0\in U_\delta(\phi_{\omega,c})$,
then the solution $u(t)$ of \eqref{eqs:gDNLS} with $u(0)=u_0$ exists for all $t>0$,
and $u(t)\in U_\varepsilon(\phi_{\omega,c})$ for all $t>0$.
Otherwise, $u_{\omega,c}$ is said to be unstable.
\end{defn}

The main result in the present paper is
\begin{thm}\label{thm:mainthm}
	Let $1<\sigma<2$ and $z_{0}=z_{0}(\sigma) \in(-1,1)$ satisfy  $F_\sigma(z_0)=0$. Then the solitary wave solutions $e^{i\omega t}\phi_{\omega,c}(x-c t)$ of \eqref{eqs:gDNLS} is unstable if $c=2z_0\sqrt\omega$.
\end{thm}

In this paper, we use the same ideas as in \cite{Yifei}. It relies on the modulation theory and construction of the virial identities. Compared to \cite{Fu-16-DNLS}, the idea is to utilize  virial identities to replace the Lyapunov functional, to obtain the lower bound on modulations. This can be used to avoid the requirement of the high-order regularities of the energy.  However, the construction in the present paper is much more delicate, due to the complicated structure of the equation.

This paper is organized as follows. In Section 2, we give the definitions of some important functionals and some useful lemmas. In Section 3, we obtain the modulation result and show the coercivity for the second variation. In Section 4,  we prove the main theorem.

\begin{remark}
We note that the same result in Theorem \ref{thm:mainthm} was obtained independently by Miao-Tang-Xu in \cite{Miao-Tang-Xu18} (appear on arXiv on March 20, 2018) by different method. They used the third derivative of the energy around the solitary wave. 
\end{remark}

\section{Preliminaries}

\subsection{Notations}

We use $A\lesssim B$ or $A=O(B)$ to denote an estimate of the form $A\leq C B$ for some constant $C>0$.
Similarly, we will write $A\sim B$ to mean $A\lesssim B$ and $B\lesssim A$. And we denote $\dot f=\partial_t f$.

For $u,v\in L^2(\mathbb{R})=L^2(\mathbb{R,C})$, we define
$$\langle u,v\rangle=\mbox{Re}\int_{\mathbb{R}}u(x)\overline{v(x)}\,dx$$
and regard $L^2(\mathbb{R})$ as a real Hilbert space.
%For $u$, $v\in L^2(\mathbb{R})$,
%let $(u,v)_{L^2}=\mbox{Re}\displaystyle{\int_{\mathbb{R}}u(x)\overline{v(x)}\,dx}$.

For a function $f(x)$, its $L^{q}$-norm $\|f\|_{L^q}=\Big(\displaystyle\int_{\mathbb{R}} |f(x)|^{q}dx\Big)^{\frac{1}{q}}$
and its $H^1$-norm $\|f\|_{H^1}=(\|f\|^2_{L^2}+\|\partial_x f\|^2_{L^2})^{\frac{1}{2}}$.

\subsection{Conservation laws}
The solution $u(t)$ of  \eqref{eqs:gDNLS} satisfies three conservation laws,
\begin{align}%\label{conservation}
{E}(u(t))={E}(u_0),\quad {P}(u(t))={P}(u_0),\quad {M}(u(t))={M}(u_0)\nonumber
\end{align}
for all $t\in[0,T_{\max})$,
where $T_{\max}$ denotes the maximal existence time of $u(t)$, and
\begin{align*}
{E}(u)=&\frac{1}{2}\| \partial_xu\|_{L^2}^2-\frac{1}{2(\sigma+1)}\mbox{Im}\int_{\mathbb{R}}|u|^{2\sigma}u\,\overline{\partial_xu}\,dx,\\
{P}(u)=&\frac{1}{2}(i\partial_xu,u)_{L^2}=\frac{1}{2}\mbox{Im}\int_{\mathbb{R}}u\,\overline{\partial_xu}\,dx,\\
{M}(u)=&\frac{1}{2}\|u\|_{L^2}^2.
\end{align*}

\subsection{Some functionals}
From the definitions of $E$, $P$ and $M$, we have
\begin{align*}
E'(u)=&-\partial_x^2u-i|u|^{2\sigma}\partial_xu,\\
P'(u)=&i\partial_xu,\\
M'(u)=&u.
\end{align*}
Let
\begin{align*}
S_{\omega,c}(u)&=E(u)+\omega M(u)+cP(u),
\end{align*}
then we have
\begin{align}
S'_{\omega,c}(u)=&E'(u)+\omega M'(u)+cP'(u)\nonumber\\
=&-\partial_x^2u-i|u|^{2\sigma}\partial_xu+\omega u+ic\partial_xu.\label{S'}
\end{align}
Hence,  \eqref{Elliptic-comp} is equivalent to $S'_{\omega,c}(\phi)=0$.
Hence for the solution $\phi_{\omega,c}$ to \eqref{Elliptic-comp}, we have
\begin{align}\label{S'wc}
S'_{\omega,c}(\phi_{\omega,c})=0.
\end{align}
Moreover, by \eqref{S'}, we obtain
\begin{align}\label{S''}
S''_{\omega,c}(\phi_{\omega,c})f=&-\partial_x^2f+\omega f+ic\partial_x f-i\sigma|\phi_{\omega,c}|^{2\sigma-2}\overline{\phi_{\omega, c}}\,\partial_x\phi_{\omega,c}\,f\nonumber\\
&-i\sigma |\phi_{\omega, c}|^{2\sigma-2}\phi_{\omega,c}\partial_x\phi_{\omega,c}\,\overline{f}-i|\phi_{\omega,c}|^{2\sigma}\partial_xf.
\end{align}	

\subsection{Useful Lemma}
In this subsection, we give some lemmas which are useful in the following sections.
First, we have following formulas.
\begin{lem}\label{S''f}
	Let $1<\sigma<2$ and $(\omega,c)\in \R^2$ satisfy $c^2<4\omega$, we have
	\begin{align}
	S''_{\omega, c}(\phi_{\omega, c})\phi_{\omega, c}=&-2\sigma i |\phi_{\omega, c}|^{2\sigma}\partial_x\phi_{\omega, c},\label{S''phi}\\ 
	S''_{\omega, c}(\phi_{\omega, c})(i\partial_x\phi_{\omega, c})=&-2\sigma\omega|\phi_{\omega, c}|^{2\sigma}\phi_{\omega, c}.\nonumber
	\end{align}
\end{lem}
\begin{proof}
First, using \eqref{S'wc} and \eqref{S''}, we get
\begin{align*}
S''_{\omega, c}(\phi_{\omega, c})\phi_{\omega, c}
=&-\partial_x^2\phi_{\omega,c}+\omega\phi_{\omega,c}+ic\partial_x\phi_{\omega,c}-i\sigma|\phi_{\omega,c}|^{2\sigma-2}|\phi_{\omega, c}|^2\,\partial_x\phi_{\omega,c}\\
&-i\sigma|\phi_{\omega,c}|^{2\sigma-2}|\phi_{\omega,c}|^2\partial_x\phi_{\omega,c}-i|\phi_{\omega,c}|^{2\sigma}\partial_x\phi_{\omega,c}\\
=&-\partial_x^2\phi_{\omega,c}-(2\sigma+1)i|\phi_{\omega,c}|^{2\sigma}\partial_x\phi_{\omega,c}+\omega\phi_{\omega,c}+ic\partial_x\phi_{\omega,c}\\
=&-2\sigma i |\phi_{\omega, c}|^{2\sigma}\partial_x\phi_{\omega, c}.
\end{align*}

Similarly, using \eqref{S'wc} and \eqref{S''},  we obtain
\begin{align*}
S''_{\omega, c}(\phi_{\omega, c})(i\partial_x\phi_{\omega, c})
%=&-\partial_x^2(i\partial_x\phi_{\omega, c})+\omega(i\partial_x\phi_{\omega, c})+ic\partial_x (i\partial_x\phi_{\omega, c})-i\sigma|\phi_{\omega,c}|^{2\sigma-2}\overline{\phi_{\omega, c}}\,\partial_x\phi_{\omega,c}(i\partial_x\phi_{\omega, c})\\
%&-i\sigma |\phi_{\omega, c}|^{2\sigma-2}\phi_{\omega,c}\partial_x\phi_{\omega,c}\,\overline{(i\partial_x\phi_{\omega, c})}-i|\phi_{\omega,c}|^{2\sigma}\partial_x(i\partial_x\phi_{\omega, c})\\
=&i\partial_x\big[-\partial_x^2\phi_{\omega,c}+\omega\phi_{\omega,c}
+ic\partial_x\phi_{\omega,c}-i\sigma|\phi_{\omega,c}|^{2\sigma}\,\partial_x\phi_{\omega,c}\big]\\
&-2\sigma|\phi_{\omega,c}|^{2\sigma-2}\,\phi_{\omega,c}|\partial_x\phi_{\omega,c}|^2\\
=&-2\sigma\omega|\phi_{\omega, c}|^{2\sigma}\phi_{\omega, c}.
\end{align*}
This concludes the proof of Lemma \ref{S''f}.
\end{proof}

Let
$$J(u)=\mbox{Im}\int_{\mathbb{R}}|u|^{2\sigma}u\,\overline{\partial_x u} dx.$$
Then we have
\begin{align}\label{J'u}
J'(u)=2(\sigma+1) i|u|^{2\sigma}\partial_x u.
\end{align}
Moreover, we have the following lemma.
\begin{lem}\label{JL}
Let $1<\sigma<2$ and $(\omega,c)\in \R^2$ satisfy $c^2<4\omega$, then
\begin{align}\label{w}
\|\partial_x\phi_{\omega, c}\|_{L^2}^2=&\omega\|\phi_{\omega, c}\|_{L^2}^2.
\end{align}
Moreover,
\begin{align}\label{J}
J(\phi_{\omega, c})=4\omega M(\phi_{\omega, c})+2cP(\phi_{\omega, c}),
\end{align}
\begin{align}\label{JPE}
\frac{\sigma-1}{\sigma+1}J(\phi_{\omega,c})=2cP(\phi_{\omega,c})+4E(\phi_{\omega,c}),
\end{align}
and
\begin{align}\label{J'}
J'(\phi_{\omega, c})=-\frac{\sigma+1}{\sigma}S''_{\omega, c}(\phi_{\omega, c})\phi_{\omega, c}.
\end{align}
\end{lem}
\begin{proof}
From the equation $S'_{\omega,c}(\phi_{\omega,c})=0$, by producting with $\overline{x\partial_x\phi_{\omega,c}}$ and $\overline{\phi_{\omega,c}}$ respectively, and taking the real part, we obtain
%$-\partial_x^2\phi_{\omega, c}+\omega\phi_{\omega, c}+ic\phi_{\omega, c}-i|\phi_{\omega, c}|^{2\sigma}\partial_x\phi_{\omega, c}=0,$
\begin{align*}
\|\partial_x\phi_{\omega, c}\|_{L^2}^2=\omega\|\phi_{\omega, c}\|_{L^2}^2,
\end{align*}
and
\begin{align*}
\|\partial_x\phi_{\omega, c}\|_{L^2}^2+\omega\|\phi_{\omega, c}\|_{L^2}^2+c\,\mbox{Im}\int_{\mathbb{R}}\phi_{\omega, c}\,\overline{\partial_x\phi_{\omega, c}} dx-J(\phi_{\omega, c})=0.
\end{align*}
Therefore, we have
$$J(\phi_{\omega, c})=4\omega M(\phi_{\omega, c})+2cP(\phi_{\omega, c}).$$
Combining the definiton of $E$ and \eqref{w}, we have
\begin{align*}
E(\phi_{\omega,c})=&\frac 12 \|\partial_x\phi_{\omega,c}\|^2_{L^2}-\frac{1}{2(\sigma+1)}\mbox{Im}\int_{\mathbb{R}}|u|^{2\sigma}u\,\overline{\partial_xu}\,dx\\
=&\omega M(\phi_{\omega,c})-\frac{1}{2(\sigma+1)}J(\phi_{\omega,c}).
\end{align*}
Then, we get
\begin{align*}
\omega M(\phi_{\omega,c})=E(\phi_{\omega,c})+\frac{1}{2(\sigma+1)}J(\phi_{\omega,c}).
\end{align*}
Hence, we obtain
\begin{align*}
J(\phi_{\omega,c})=&4[E(\phi_{\omega,c})+\frac{1}{2(\sigma+1)}J(\phi_{\omega,c})]+2cP(\phi_{\omega,c})\\
=&2cP(\phi_{\omega,c})+4E(\phi_{\omega,c})+\frac{2}{\sigma+1}J(\phi_{\omega,c}).
\end{align*}
That is,
\begin{align*}
\frac{\sigma-1}{\sigma+1}J(\phi_{\omega,c})=2cP(\phi_{\omega,c})+4E(\phi_{\omega,c}).
\end{align*}
Moreover, from \eqref{S''phi} and \eqref{J'u}, we have
\begin{align*}
J'(\phi_{\omega, c})=&2(\sigma+1) i|\phi_{\omega, c}|^{2\sigma}\partial_x \phi_{\omega, c}\\
=&-\frac{\sigma+1}{\sigma}S''_{\omega, c}(\phi_{\omega, c})\phi_{\omega, c}.
\end{align*}
This completes the proof.
\end{proof}

\begin{lem}\label{Apendix1}
Let $1<\sigma<2$ and $(\omega,c)\in \R^2$ satisfy $c^2<4\omega$, then
\begin{align}\label{xphi}
\|\phi_{\omega,c}\|_{L^{2\sigma+2}}^{2\sigma+2}=4(\sigma+1)\big[\frac c2M(\phi_{\omega,c})+P(\phi_{\omega,c})\big],
\end{align}
and
$$\partial_cM(\phi_{\omega,c})=\partial_\omega P(\phi_{\omega,c}), \quad   \partial_cP(\phi_{\omega,c})=\omega\partial_\omega M(\phi_{\omega,c}).$$
%If $c=2z_0\sqrt\omega$, we have
%$$P(\phi_{\omega,c})=a_0M(\phi_{\omega,c}),\quad  \frac \mu\nu=\sqrt\omega,$$
%where $a_0=(\sigma-1)\sqrt\omega$.
%Moreover,
%\begin{align}
%\mu^2\partial_{\omega \omega}M+2\mu\nu\partial_{\omega c}M+\nu^2\partial_{c c}M
%=&8\nu^2\sqrt\omega\tilde\kappa_\omega\alpha_0(\sigma-1)z_0,\label{uvM}\\
%\mu^2\partial_{\omega \omega}P+2\mu\nu\partial_{\omega c}P+\nu^2\partial_{c c}P
%=&-8\nu^2\omega\tilde{\kappa}_\omega\alpha_0(\sigma-1).\label{uvP}
%\end{align}
\end{lem}
\begin{proof}
The details are given in Appendix.
\end{proof}

For any $(\omega,c)\in \R^2$ satisfying $c^2<4\omega$, we define a function $d(\omega,c)$ by
$$d(\omega,c)=S_{\omega,c}(\phi_{\omega,c}).$$
Thus, we have
$$d'(\omega,c)=\big(\partial_\omega d(\omega,c), \partial_c d(\omega,c)\big)=\big(M(\phi_{\omega,c}), P(\phi_{\omega,c})\big),$$
and the Hessian matrix $d''(\omega,c)$ of $d(\omega,c)$ is given by
\begin{align*}
d''(\omega,c)=
\left[\begin{array}{ccccc}
\partial_{\omega\omega} d(\omega,c) & \partial_{\omega c}d(\omega,c)\\
\partial_{\omega c} d(\omega,c)&\partial_{cc} d(\omega,c)
\end{array}\right]
=\left[\begin{array}{ccccc}
\partial_\omega M(\phi_{\omega,c}) &\partial_\omega P(\phi_{\omega,c})\\
\partial_c M(\phi_{\omega,c}) &\partial_c P(\phi_{\omega,c})
\end{array}\right].
\end{align*}

For general exponents $1<\sigma<2$, Liu, Simpson and Sulem \cite{LiSiSu1} proved that $z_0$ is the unique solution of $\det[d''(\omega,c)]=0$.

Let $(\mu, \nu)$ to be the eigenvector associated to zero eigenvalue of the Hessian matrix $d''(\omega,c)$. Since zero is the simple eigenvalue, $(\mu, \nu)$ is unique up to a constant. That is,
\begin{equation}\label{0}
\left\{ \aligned
&\mu\,\partial_\omega M(\phi_{\omega,c})+\nu\,\partial_\omega P(\phi_{\omega,c})=0,\\
&\mu\,\partial_c M(\phi_{\omega,c})+\nu\,\partial_c P(\phi_{\omega,c})=0.
\endaligned
\right.
\end{equation}

Together with $\partial_cM(\phi_{\omega,c})=\partial_\omega P(\phi_{\omega,c})$, \eqref{0} is equivalent to
\begin{equation}\label{0'}
\left\{ \aligned
&\mu\,\partial_\omega M(\phi_{\omega,c})+\nu\,\partial_c M(\phi_{\omega,c})=0,\\
&\mu\,\partial_\omega P(\phi_{\omega,c})+\nu\,\partial_c P(\phi_{\omega,c})=0.
\endaligned
\right.
\end{equation}

Now we have the following lemma.
\begin{lem}\label{Apendix}
Let $1<\sigma<2$ and $(\omega,c)\in \R^2$ satisfy $c=2z_0\sqrt\omega$, then
$$P(\phi_{\omega,c})=a_0M(\phi_{\omega,c}),\quad  \frac \mu\nu=\sqrt\omega,$$
where $a_0=(\sigma-1)\sqrt\omega>0$.
Moreover, there exists $\kappa_0>0$, such that
\begin{align}
\mu^2\partial_{\omega \omega}M(\phi_{\omega,c})+2\mu\nu\partial_{\omega c}M(\phi_{\omega,c})+\nu^2\partial_{c c}M(\phi_{\omega,c})
=&\kappa_0z_0,\label{uvM}\\
\mu^2\partial_{\omega \omega}P(\phi_{\omega,c})+2\mu\nu\partial_{\omega c}P(\phi_{\omega,c})+\nu^2\partial_{c c}P(\phi_{\omega,c})
=&-\kappa_0\sqrt{\omega}.\label{uvP}
\end{align}
\end{lem}
\begin{proof}
The proof of Lemma \ref{Apendix} is postponed to Appendix.
\end{proof}

For convinience, we denote the quality $Q_{\mu, \nu}$ to be:
$$Q_{\mu, \nu}(f)=\mu M(f)+\nu P(f).$$
%where $(\mu, \nu)$ to be the eigenvector associated to zero eigenvalue of the Hessian matrix $d''(\omega,c)$.
Moreover, we denote $\psi, \tilde\psi$ as
\begin{align*}
\psi=&\partial_\lambda \phi_{\omega+\lambda\mu, c+\lambda\nu}\lvert_{\lambda=0}=\mu \partial_\omega\phi_{\omega, c}+\nu\partial_c\phi_{\omega, c},\\
\tilde\psi=&\frac 12\partial_\lambda^2 \phi_{\omega+\lambda\mu, c+\lambda\nu}\lvert_{\lambda=0}.
\end{align*}
\begin{lem}\label{critical}
Let $1<\sigma<2$. If $c=2z_0\sqrt\omega$, then
\begin{align}\label{M'P'}
\langle M'(\phi_{\omega, c}), \psi\rangle=\langle P'(\phi_{\omega, c}), \psi\rangle=0,
\end{align}
and
\begin{align}\label{Q'0}
S''_{\omega, c}(\phi_{\omega, c})\psi=-Q'_{\mu, \nu}(\phi_{\omega, c}),\quad \big\langle S''_{\omega, c}(\phi_{\omega, c})\psi,\psi\big\rangle=0.
\end{align}
\end{lem}
\begin{proof}
By \eqref{0'}, we get
\begin{equation*}
\left\{ \aligned
&\langle M'(\phi_{\omega, c}), \mu\partial_\omega\phi_{\omega, c}+\nu\partial_c\phi_{\omega, c}\rangle=0,\\
&\langle P'(\phi_{\omega, c}), \mu\partial_\omega\phi_{\omega, c}+\nu\partial_c\phi_{\omega, c})\rangle=0,
\endaligned
\right.
\end{equation*}
that is,
$$\langle M'(\phi_{\omega, c}), \psi\rangle=\langle P'(\phi_{\omega, c}), \psi\rangle=0.$$
From \eqref{0}, we have
\begin{equation*}
\left\{ \aligned
&\langle\mu M'(\phi_{\omega, c})+\nu P'(\phi_{\omega, c}), \partial_\omega\phi_{\omega, c}\rangle=0,\\
&\langle\mu M'(\phi_{\omega, c})+\nu P'(\phi_{\omega, c}), \partial_c\phi_{\omega, c}\rangle=0.
\endaligned
\right.
\end{equation*}
Therefore, we have
$$\langle Q'_{\mu, \nu}(\phi_{\omega, c}), \psi\rangle=0.$$
On the other hand, differentiating	$S'_{\omega+\lambda\mu, c+\lambda\nu}(\phi_{\omega+\lambda\mu,c+\lambda\nu})=0$ with respect to $\lambda=0$, we have
\begin{align*}
S''_{\omega+\lambda\mu, c+\lambda\nu}(\phi_{\omega+\lambda\mu,c+\lambda\nu})&\partial_\lambda\phi_{\omega+\lambda\mu,c+\lambda\nu}\big|_{\lambda=0}\\
=&-\big[\mu M'(\phi_{\omega+\lambda\mu,c+\lambda\nu})+\nu P'(\phi_{\omega+\lambda\mu,c+\lambda\nu})\big]\big|_{\lambda=0}.
\end{align*}
That is
$$S''_{\omega, c}(\phi_{\omega, c})\psi=-Q'_{\mu, \nu}(\phi_{\omega, c}).$$
Thus, we have
$$\quad \big\langle S''_{\omega, c}(\phi_{\omega, c})\psi,\psi\big\rangle=-\langle Q'_{\mu, \nu}(\phi_{\omega, c}), \psi\rangle=0.$$
This proves the lemma.
\end{proof}

\begin{lem}\label{J'psi}
Let $1<\sigma<2$. If $c=2z_0\sqrt\omega$,  then
\begin{align*}
\langle J'(\phi_{\omega, c}), \psi \rangle=4\mu M(\phi_{\omega, c})+2\nu P(\phi_{\omega, c})\neq 0.
\end{align*}
\end{lem}
\begin{proof}
Note that $\psi=\partial_\lambda \phi_{\omega+\lambda\mu, c+\lambda\nu}\lvert_{\lambda=0}$, using \eqref{J}, we can write
\begin{align*}
\langle J'(\phi_{\omega, c}), \psi \rangle =&\partial_\lambda J(\phi_{\omega+\lambda\mu, c+\lambda\nu})\lvert_{\lambda=0}\\
=&\partial_\lambda\big[4(\omega+\lambda\mu)M(\phi_{\omega+\lambda\mu, c+\lambda\nu})+2(c+\lambda\nu)P(\phi_{\omega+\lambda\mu, c+\lambda\nu})\big]\big\lvert_{\lambda=0}\\
=&4\mu M(\phi_{\omega, c})+2\nu P(\phi_{\omega, c})
+\langle M'(\phi_{\omega, c}), \partial_\lambda\phi_{\omega+\lambda\mu, c+\lambda\nu}\lvert_{\lambda=0} \rangle\\
&+\langle P'(\phi_{\omega, c}), \partial_\lambda\phi_{\omega+\lambda\mu, c+\lambda\nu}\lvert_{\lambda=0} \rangle .
\end{align*}
When $c=2z_0\sqrt\omega$, together with \eqref{M'P'}, we obtain
\begin{align*}
\langle J'(\phi_{\omega, c}), \psi \rangle
=&4\mu M(\phi_{\omega, c})+2\nu P(\phi_{\omega, c})+\langle M'(\phi_{\omega, c}), \psi \rangle+\langle P'(\phi_{\omega, c}), \psi \rangle\\
=&4\mu M(\phi_{\omega, c})+2\nu P(\phi_{\omega, c}).
\end{align*}
Using Lemma \ref{Apendix}, we find
$$
4\mu M(\phi_{\omega, c})+2\nu P(\phi_{\omega, c})=\nu \big(4\sqrt \omega+2a_0\big)M(\phi_{\omega, c})
\neq 0.
$$
Hence, $\langle J'(\phi_{\omega, c}), \psi \rangle\neq 0$.
This finishes the proof.
\end{proof}

\section{Modulation and the coercivity property}
%Denote
%\begin{align*}
%\varPsi_{\omega, c}=\frac{P(\phi_{\omega, c})}{(a_0^2-\omega)M(\phi_{\omega,c})}\, S''_{\omega, c}(\phi_{\omega,c})(a_0\phi_{\omega,c}+i\partial_x\phi_{\omega,c})
%+2(\sigma-1)\partial_x^2\phi_{\omega,c},
%\end{align*}
%where $a_0=(\sigma-1)\sqrt\omega.$

\begin{prop}\label{Modulation}
There exists $\delta_0>0$, such that for any $\delta \in (0, \delta_0)$, $u\in U_\delta(\phi_{\omega,c})$, the following properties is verified. There exist $C^1$-functions
$$(\theta, y, \lambda):U_\delta(\phi_{\omega,c}) \rightarrow \R\times\R\times\R^+,$$
such that if we define $\varepsilon(t)$ by
\begin{align}\label{modulation-u}
\varepsilon(t)=e^{-i\theta(t)}u(t,\,\cdot+y(t))-\phi_{\omega+\lambda(t)\mu, c+\lambda(t)\nu},
\end{align}
then $\varepsilon$ satisfies the following orthogonality conditions for any $t\in\R$,
\begin{align*}%\label{orth-condition}
\langle \varepsilon, i\phi_{\omega+\lambda\mu,c+\lambda\nu}\rangle
=\langle \varepsilon, \partial_x\phi_{\omega+\lambda\mu,c+\lambda\nu} \rangle
=\langle \varepsilon, J'(\phi_{\omega+\lambda\mu,c+\lambda\nu})\rangle=0.
\end{align*}
%where $\phi_{\overrightarrow{\lambda}}=\phi_{\omega+\lambda(t)\mu, c+\lambda(t)\nu}.$
Moreover,
\begin{align}\label{coercivity}
\big\langle S''_{\omega+\lambda\mu, c+\lambda\nu}(\phi_{\omega+\lambda\mu, c+\lambda\nu})\varepsilon, \varepsilon\big\rangle
\gtrsim\|\varepsilon\|_{\H1}^2.
\end{align}
\end{prop}

\begin{proof}
The proof of the proposition can be split into the following  three steps.

Step 1, modulation for fixed time.  Fixing $t\in\R$, let $\overrightarrow F(\theta, y, \lambda; u)=(F_1, F_2, F_3)$ with
\begin{align*}
F_1(\theta, y, \lambda; u)=&\langle \varepsilon, i\phi_{\omega+\lambda\mu,c+\lambda\nu}\rangle ,\\
F_2(\theta, y, \lambda; u)=&\langle \varepsilon, \partial_x\phi_{\omega+\lambda\mu,c+\lambda\nu} \rangle,\\
F_3(\theta, y, \lambda; u)=&\langle \varepsilon, J'(\phi_{\omega+\lambda\mu,c+\lambda\nu})\rangle.
\end{align*}
Note that
\begin{align*}
\partial_\theta\varepsilon\big|_{(0,0,0; \phi_{\omega, c})}=-i\phi_{\omega, c};\quad
\partial_y\varepsilon\big|_{(0,0,0; \phi_{\omega, c})}=\partial_x\phi_{\omega, c};\quad
\partial_\lambda\varepsilon\big|_{(0,0,0; \phi_{\omega, c})}=-\psi.
\end{align*}
Then, the Jacobian matrix of the derivative of the function $(\theta, y, \lambda; u)\mapsto \overrightarrow F$ with respect to $(\theta,y,\lambda)$ is as follows.
\begin{align*}%\label{eq:jacobian}
D\overrightarrow F\Big|_{(0,0,0; \phi_{\omega, c})}
&=\begin{pmatrix}
\partial_\theta F_1&\partial_y F_1&\partial_\lambda F_1\\
\partial_\theta F_2&\partial_y F_2&\partial_\lambda F_2\\
\partial_\theta F_3&\partial_y F_3&\partial_\lambda F_3\\
\end{pmatrix}\Biggl|_{(0,0,0; \phi_{\omega, c})}\notag\\
&=\begin{pmatrix}
-\|\phi_{\omega, c}\|_{L^2}^2&-2P(\phi_{\omega, c})&-\langle \psi, i\phi_{\omega, c}\rangle\\
2P(\phi_{\omega, c})&\|\partial_x\phi_{\omega, c}\|_{L^2}^2&-\langle \psi, \partial_x\phi_{\omega, c}\rangle\\
0&0&-\langle J'(\phi_{\omega,c}), \psi\rangle
\end{pmatrix}.
\end{align*}
%According to the definition of $\varPsi_{\omega, c}$, $J(\phi_{\omega, c})$ and \eqref{Elliptic-comp}, we have
%\begin{align*}
%\langle \psi, |\phi_{\omega,c}|^{2\sigma}\phi_{\omega,c}\rangle=&\frac 12(\mu\partial_\omega+\nu\partial_c)\|\phi_{\omega,c}\|_{L^{2\sigma+2}}^{2\sigma+2}\\
%=&2(\sigma+1)(\mu\partial_\omega+\nu\partial_c)\big(\frac c2M(\phi_{\omega, c})+P(\phi_{\omega, c})\big)\\
%=&(\sigma+1)\nu M(\phi_{\omega, c})\\
%\neq& 0
%\end{align*}
Thus, we can get
$$\det{(D\overrightarrow F)}\Big|_{(0,0,0; \phi_{\omega, c})}=4\sigma(2-\sigma)\omega [M(\phi_{\omega, c})]^2\langle J'(\phi_{\omega,c}), \psi\rangle.$$
From Lemma \ref{J'psi}, we have
$$\det{(D\overrightarrow F)}\Big|_{(0,0,0; \phi_{\omega, c})}\neq0.$$
Therefore, the implicit function theorem implies that there exists $\delta_0>0$, such that for any $\delta \in (0, \delta_0)$, $u\in U_\delta(\phi_{\omega,c})$, the following properties is verified. There exist continuity functions
$$(\theta, y, \lambda):U_\delta(\phi_{\omega,c}) \rightarrow \R\times\R\times\R^+,$$
such that $F_j(\theta, y, \lambda; u)=0$, $j=1, 2, 3.$

Step 2, the regularity of the parameters in time. The parameters $(\theta, y, \lambda)\in C^1_t$ can be followed from the regularization arguments.

Step 3, the coercivity property of $S''_{\omega, c}(\phi_{\omega, c})$. From \cite{LiSiSu1} Theorem 3.1, we obtain that $S''_{\omega, c}(\phi_{\omega, c})$ has exactly one negative eigenvalue. Hence there exists only one $\lambda_{-1}<0$, such that,
$$S''_{\omega, c}(\phi_{\omega, c})g_{-1}=\lambda_{-1}g_{-1}, \quad \|g_{-1}\|_{L^2}=1.$$
Moreover, we have the following decomposition,
$$\varepsilon=a_{-1}g_{-1}+a_1i\phi_{\omega, c}+a_2\partial_x\phi_{\omega, c}+h_1,$$
with $$\langle h_1, g_{-1}\rangle=\langle h_1, i\phi_{\omega, c}\rangle=\langle h_1,\partial_x\phi_{\omega, c}\rangle=0,$$ and
$$\langle S''_{\omega, c}(\phi_{\omega, c})h_1, h_1\rangle\gtrsim\|h_1\|_{\H1}^2.$$
Since $\langle \varepsilon, i\phi_{\omega,c}\rangle=\langle \varepsilon, \partial_x\phi_{\omega,c}\rangle=0$, we have $a_1=a_2=0$. Then, we can write
$$\varepsilon=a_{-1}g_{-1}+h_1.$$
Next, using \eqref{J'}, we have
$$J'(\phi_{\omega,c})=S''_{\omega,c}(\phi_{\omega,c})(-\frac{\sigma+1}{\sigma}\phi_{\omega,c}).$$
For convenience, we put $$h=-\frac{\sigma+1}{\sigma}\phi_{\omega,c}.$$
Note that $\langle h, i\phi_{\omega,c}\rangle=\langle h, \partial_x\phi_{\omega,c}\rangle=0$, we can also write that
$$h=b_{-1}g_{-1}+h_2$$
with $$\langle h_2, g_{-1}\rangle=\langle h_2, i\phi_{\omega, c}\rangle=\langle h_2,\partial_x\phi_{\omega, c}\rangle=0,$$ and
$$\langle S''_{\omega, c}(\phi_{\omega, c})h_2, h_2 \rangle\gtrsim\|h_2\|_{\H1}^2.$$
For simplicity, we denote
\begin{align*}
\gamma=-\langle S''_{\omega,c}(\phi_{\omega,c})h,h\rangle=-\langle J'(\phi_{\omega,c}), h\rangle=\frac{2(\sigma+1)^2}{\sigma}J(\phi_{\omega,c}).
\end{align*}
Then, from \eqref{J} and Lemma \ref{Apendix}, we know $\gamma>0$.

Moreover, we have
\begin{align}
\big\langle S''_{\omega,c}(\phi_{\omega, c})\varepsilon, \varepsilon\big\rangle=&\lambda_{-1}a_{-1}^2+\langle S''_{\omega, c}(\phi_{\omega, c})h_1, h_1\rangle,\label{B}\\
\big\langle S''_{\omega,c}(\phi_{\omega, c})h, h\big\rangle=&\lambda_{-1}b_{-1}^2+\langle S''_{\omega, c}(\phi_{\omega, c})h_2, h_2\rangle=-\gamma<0.\label{A}
\end{align}
According to the orthogonality condition $\langle \varepsilon, J'(\phi_{\omega,c})\rangle=0$ and some direct computations, we have
\begin{align}\label{h1h2}
\lambda_{-1}a_{-1}b_{-1}+\langle S''_{\omega, c}(\phi_{\omega, c})h_1, h_2\rangle=0.
\end{align}
Together with \eqref{B}, \eqref{A}, \eqref{h1h2} and the Cauchy-Schwartz inequlity, we obtain that
\begin{align}
-\lambda_{-1}a_{-1}^2=&\frac{\lambda_{-1}^2a_{-1}^2b_{-1}^2}{-\lambda_{-1}b_{-1}^2}=\frac{\langle S''_{\omega, c}(\phi_{\omega, c})h_1, h_2\rangle^2}{-\lambda_{-1}b_{-1}^2}\notag\\
\leq&\frac{\langle S''_{\omega, c}(\phi_{\omega, c})h_1, h_2\rangle^2}{\gamma+\langle S''_{\omega, c}(\phi_{\omega, c})h_2, h_2\rangle}
\leq \frac{\langle S''_{\omega, c}(\phi_{\omega, c})h_1, h_1\rangle \langle S''_{\omega, c}(\phi_{\omega, c})h_2, h_2\rangle}{\gamma+\langle S''_{\omega, c}(\phi_{\omega, c})h_2, h_2\rangle}.\label{14.15}
\end{align}
Thus, we get
\begin{align*}
\big\langle S''_{\omega,c}(\phi_{\omega, c})\varepsilon, \varepsilon\big\rangle\geq&-\frac{\langle S''_{\omega, c}(\phi_{\omega, c})h_1, h_1\rangle \langle S''_{\omega, c}(\phi_{\omega, c})h_2, h_2\rangle}{\gamma+\langle S''_{\omega, c}(\phi_{\omega, c})h_2, h_2\rangle}+\langle S''_{\omega, c}(\phi_{\omega, c})h_1, h_1\rangle\\
=&\frac{\gamma}{\gamma+\langle S''_{\omega, c}(\phi_{\omega, c})h_2, h_2\rangle}\langle S''_{\omega, c}(\phi_{\omega, c})h_1, h_1\rangle\\
%\gtrsim&\langle S''_{\omega, c}(\phi_{\omega, c})h_1, h_1\rangle\\
\gtrsim&\|h_1\|_{\H1}^2.
\end{align*}
By \eqref{14.15} and H\"{o}lder's inequality,  we have
$$a_{-1}^2\lesssim\|h_1\|_{\H1}^2.$$
Hence, by $\varepsilon=a_{-1}g_{-1}+h_1$, we have
$$\|\varepsilon\|_{L^2}^2\lesssim a_{-1}^2+\|h_1\|_{\H1}^2\lesssim\|h_1\|_{\H1}^2.$$
Since,
$$\big\langle S''_{\omega,c}(\phi_{\omega, c})\varepsilon, \varepsilon\big\rangle\gtrsim\|h_1\|_{\H1}^2\gtrsim\|\varepsilon\|_{L^2}^2.$$
From the definition of $S''_{\omega,c}(\phi_{\omega,c})$ in \eqref{S''}, we have
$$\|\varepsilon\|_{\dot H^1}^2\lesssim \big\langle S''_{\omega,c}(\phi_{\omega, c})\varepsilon, \varepsilon\big\rangle+\|\varepsilon\|_{L^2}^2.$$
Therefore, we get that
$$\big\langle S''_{\omega, c}(\phi_{\omega, c})\varepsilon, \varepsilon\big\rangle
\gtrsim\|\varepsilon\|_{\H1}^2.$$
This finishes the proof of the proposition.
\end{proof}

\begin{lem}\label{doty}
There exists $C_{\omega,c}\in \R$, such that
\begin{align}\label{est-derivative}
\dot y-c-\lambda\nu=C_{\omega,c}\dot\lambda+\big\langle S''_{\omega, c}(\phi_{\omega, c})\frac{-a_0\phi_{\omega, c}+i\partial_x\phi_{\omega, c}}{2(a_0^2-\omega)M(\phi_{\omega, c})}, \varepsilon\big\rangle+O(\lambda\|\varepsilon\|_{\H1}+\|\varepsilon\|^2_{\H1}),
\end{align}
and
\begin{align}\label{dotl}
\dot\theta-\omega-\lambda\mu
=O(\|\varepsilon\|_{\H1}),\quad \dot\lambda=O(\|\varepsilon\|_{\H1}).
\end{align}
Here the parameters $\theta$, $y$, $\lambda$ are given by Proposition \ref{Modulation}.
\end{lem}

\begin{proof}
Now, we consider the dynamic of the parameters. From \eqref{modulation-u}, we have
$$u=e^{i\theta(t)}\big(\phi_{\omega+\lambda(t)\mu, c+\lambda(t)\nu}+\varepsilon(t)\big)(x-y(t)).$$
Using \eqref{eqs:gDNLS}, we obtain
\begin{align}\label{dot}
i\dot\varepsilon&-(\dot\theta-\omega-\lambda\mu)(\phi_{\omega+\lambda\mu, c+\lambda\nu}+\varepsilon)
-(\dot y-c-\lambda\nu)(i\partial_x\phi_{\omega+\lambda\mu, c+\lambda\nu}+i\partial_x\varepsilon)\nonumber\\
&+\dot \lambda\cdot i\partial_\lambda\phi_{\omega+\lambda\mu, c+\lambda\nu}
=S''_{\omega, c}(\phi_{\omega, c})\varepsilon+O(\lambda\varepsilon+\varepsilon^2),
\end{align}
where $O(\varepsilon)$ is a functional of $\varepsilon$ with the order equal or more than one.

First, the equation \eqref{dot} producting with $-a_0\phi_{\omega+\lambda\mu, c+\lambda\nu}+i\partial_x\phi_{\omega+\lambda\mu,c+\lambda\nu}$, we get
\begin{align*}
&-(\dot\theta-\omega-\lambda\mu)\big[-2a_0M(\phi_{\omega+\lambda\mu,c+\lambda\nu})+2P(\phi_{\omega+\lambda\mu,c+\lambda\nu})+O(\|\varepsilon\|_{\H1})\big]\\
&-(\dot y-c-\lambda\nu)\big[-2a_0P(\phi_{\omega+\lambda\mu,c+\lambda\nu})+\|\partial_x\phi_{\omega+\lambda\mu,c+\lambda\nu}\|_{L^2}^2+O(\|\varepsilon\|_{\H1})\big]\\
&+\dot\lambda\big[\big\langle i\partial_\lambda\phi_{\omega+\lambda\mu,c+\lambda\nu},  -a_0\phi_{\omega+\lambda\mu,c+\lambda\nu}+i\partial_x\phi_{\omega+\lambda\mu,c+\lambda\nu}\big\rangle
-\langle \varepsilon, ia_0\partial_\lambda\phi_{\omega+\lambda\mu,c+\lambda\nu}+\partial_\lambda\partial_x\phi_{\omega+\lambda\mu,c+\lambda\nu}\rangle\big]\\
&=\big\langle S''_{\omega, c}(\phi_{\omega, c})(-a_0\phi_{\omega+\lambda\mu, c+\lambda\nu}+i\partial_x\phi_{\omega+\lambda\mu,c+\lambda\nu}), \varepsilon\big\rangle
+O(\lambda\|\varepsilon\|_{\H1}+\|\varepsilon\|_{\H1}^2)
\end{align*}
By a direct expansion, we have
\begin{align}\label{T1}
\phi_{\omega+\lambda\mu, c+\lambda\nu}=\phi_{\omega, c}+\lambda\psi+O(\lambda^2),
\end{align}
and
\begin{align*}
\partial_\lambda\phi_{\omega+\lambda\mu,c+\lambda\nu}
=&\psi+O(\lambda).
\end{align*}
Moreover, together with \eqref{M'P'}, we have
\begin{align*}
M(\phi_{\omega+\lambda\mu,c+\lambda\nu})=&M(\phi_{\omega,c})+\lambda\langle M'(\phi_{\omega, c}), \psi\rangle+O(\lambda^2)=M(\phi_{\omega,c})+O(\lambda^2),\\
P(\phi_{\omega+\lambda\mu,c+\lambda\nu})=&P(\phi_{\omega,c})+\lambda\langle P'(\phi_{\omega, c}), \psi\rangle+O(\lambda^2)=P(\phi_{\omega,c})+O(\lambda^2).
\end{align*}
From \eqref{w}, we get
\begin{align*}
\|\partial_x\phi_{\omega+\lambda\mu,c+\lambda\nu}\|_{L^2}^2=&(\omega+\lambda\mu)\|\phi_{\omega+\lambda\mu,c+\lambda\nu}\|_{L^2}^2\\
=&2\omega M(\phi_{\omega,c})+O(\lambda).
\end{align*}
We collect the above computations and obtain
\begin{align*}
&-(\dot\theta-\omega-\lambda\mu)\big[-2a_0M(\phi_{\omega,c})+2P(\phi_{\omega,c})+O(\lambda^2+\|\varepsilon\|_{\H1})\big]\\
&-(\dot y-c-\lambda\nu)\big[-2a_0P(\phi_{\omega,c})+2\omega M(\phi_{\omega,c})+O(\lambda+\|\varepsilon\|_{\H1})\big]\\
&+\dot\lambda\big[\big\langle i\psi,  -a_0\phi_{\omega,c}+i\partial_x\phi_{\omega,c}\big\rangle+O(\lambda+\|\varepsilon\|_{\H1})\big]\\
&=\big\langle S''_{\omega, c}(\phi_{\omega, c})(-a_0\phi_{\omega, c}+i\partial_x\phi_{\omega, c}+O(\lambda)), \varepsilon\big\rangle
+O(\lambda\|\varepsilon\|_{\H1}+\|\varepsilon\|_{\H1}^2).
\end{align*}
By Lemma \ref{Apendix}, we know that
$$
-a_0M(\phi_{\omega,c})+P(\phi_{\omega,c})=0, \quad
-a_0P(\phi_{\omega,c})+\omega M(\phi_{\omega,c})=(\omega-a_0^2)M(\phi_{\omega, c})\neq0.
$$
Then, we get
\begin{align}\label{1}
&-(\dot\theta-\omega-\lambda\mu)\big[O(\lambda^2+\|\varepsilon\|_{\H1})\big]\nonumber\\
&-(\dot y-c-\lambda\nu)\big[2(\omega-a_0^2)M(\phi_{\omega, c})+O(\lambda+\|\varepsilon\|_{\H1})\big]\nonumber\\
&+\dot\lambda\big[\big\langle i\psi,  -a_0\phi_{\omega,c}+i\partial_x\phi_{\omega,c}\big\rangle+O(\lambda+\|\varepsilon\|_{\H1})\big]\nonumber\\
&=\big\langle S''_{\omega, c}(\phi_{\omega, c})(-a_0\phi_{\omega, c}+i\partial_x\phi_{\omega, c}), \varepsilon\big\rangle
+O(\lambda\|\varepsilon\|_{\H1}+\|\varepsilon\|_{\H1}^2).
\end{align}
%\begin{align*}
%2(a_0^2-\omega)M(\phi_{\omega, c})(\dot y-c-\lambda\nu)
%=&-\dot\lambda\langle i\psi, -a_0\phi_{\omega,c}+i\partial_x\phi_{\omega,c}\rangle\\
%&+\Big\langle S''_{\omega, c}(\phi_{\omega, c})(-a_0\phi_{\omega, c}+i\partial_x\phi_{\omega, c}), \varepsilon\Big\rangle+O(\lambda\|\varepsilon\|_{\H1}+\|\varepsilon\|_{\H1}^2).
%\end{align*}

Next, producting with $\phi_{\omega+\lambda\mu,c+\lambda\nu}$ in the equation \eqref{dot} and treating as above, 
\begin{align}\label{2}
&-(\dot\theta-\omega-\lambda\mu)[2M(\phi_{\omega, c})+O(\lambda^2+\|\varepsilon\|_{\H1})]\nonumber\\
&-(\dot y-c-\lambda\nu)[2P(\phi_{\omega,c})+O(\lambda^2+\|\varepsilon\|_{\H1})]\nonumber\\
&+\dot\lambda[(i\psi, \phi_{\omega,c})+O(\lambda+\|\varepsilon\|_{\H1})]\nonumber\\
&=O(\|\varepsilon\|_{\H1}).
\end{align}

Finally, producting with $iJ'(\phi_{\omega+\lambda\mu,c+\lambda\nu})$  in the equation \eqref{dot} and treating as above, we obtain
\begin{align}\label{3}
&-(\dot\theta-\omega-\lambda\mu)[\langle J'(\phi_{\omega,c}), -i\phi_{\omega,c}\rangle+O(\lambda+\|\varepsilon\|_{\H1})]\nonumber\\
&-(\dot y-c-\lambda\nu)[\langle J'(\phi_{\omega,c}), \partial_x\phi_{\omega,c}\rangle+O(\lambda+\|\varepsilon\|_{\H1})]\nonumber\\
&+\dot\lambda[\langle J'(\phi_{\omega,c}), \psi\rangle+O(\lambda+\|\varepsilon\|_{\H1})\big]\nonumber\\
&=O(\|\varepsilon\|_{\H1}).
\end{align}
Using \eqref{1}, \eqref{2}, \eqref{3} and Lemma \ref{J'psi}, we obtain
\begin{align*}
\dot y-c-\lambda\nu=C_{\omega,c}\dot\lambda+\big\langle S''_{\omega, c}(\phi_{\omega, c})\frac{-a_0\phi_{\omega, c}+i\partial_x\phi_{\omega, c}}{2(a_0^2-\omega)M(\phi_{\omega, c})}, \varepsilon\big\rangle+O(\lambda\|\varepsilon\|_{\H1}+\|\varepsilon\|^2_{\H1}),
\end{align*}
and
$$
\dot\theta-\omega-\lambda\mu
=O(\|\varepsilon\|_{\H1}),\quad
\dot\lambda
=O(\|\varepsilon\|_{\H1}).
$$
Here
$$
C_{\omega,c}=\frac{-1}{2(a_0^2-\omega)M(\phi_{\omega, c})}\langle i\psi, -a_0\phi_{\omega,c}+i\partial_x\phi_{\omega,c}\rangle.
$$
This completes the proof.

\end{proof}

\section{Instability of the Solitary Wave Solutions}

\subsection{Virial Estimates}
To prove the main theorem, first, we need the localized virial idetities.
\begin{lem}\label{V}
Let $\varphi\in C^\infty(\mathbb{R})$, then
\begin{align*}
\frac{d}{dt}\int_{\mathbb{R}}\varphi |u|^2 dx
=&-2\mbox{Im}\int_{\mathbb{R}}\varphi'u\,\overline{\partial_x u} dx+\frac{1}{\sigma+1}\int_{\mathbb{R}}\varphi' |u|^{2\sigma+2} dx,\\
\frac{d}{dt}\int_{\mathbb{R}}\varphi\mbox{Im}(u\,\overline{\partial_xu})
=&-2\int_{\mathbb{R}}\varphi'|\partial_x u|^2 dx+\frac 12\int_{\mathbb{R}}\varphi'''|u|^2 dx+\mbox{Im}\int_{\mathbb{R}}\varphi'|u|^{2\sigma}u\,\overline{\partial_xu}.
\end{align*}
\end{lem}
\begin{proof}
Combining the eqution \eqref{eqs:gDNLS} and integration by parts,  we have
\begin{align*}
\frac{d}{dt}\int_{\mathbb{R}}\varphi |u|^2 dx
=&2\mbox{Re}\int_{\mathbb{R}}\varphi\partial_tu\,\overline u\\
=&2\mbox{Re}\int_{\mathbb{R}}\varphi(i\partial_x^2u-|u|^{2\sigma}\partial_xu)\overline u\\
=&-2\mbox{Re}\int_{\mathbb{R}}i\varphi'\partial_xu\,\overline u dx+\frac{1}{\sigma+1}\mbox{Re}\int_{\mathbb{R}}\varphi|u|^{2\sigma+2}dx\\
=&-2\mbox{Im}\int_{\mathbb{R}}\varphi'u\,\overline{\partial_x u} dx+\frac{1}{\sigma+1}\int_{\mathbb{R}}\varphi' |u|^{2\sigma+2} dx.
\end{align*}
By the same way, we obtain
\begin{align*}
\frac{d}{dt}\int_{\mathbb{R}}\varphi\mbox{Im}(u\,\overline{\partial_xu})
=&\mbox{Im}\int_{\mathbb{R}}\varphi\partial_tu\,\overline {\partial_xu} dx+\mbox{Im}\int_{\mathbb{R}}\varphi u\,\overline{\partial_{t x}u}dx\\
=&2\mbox{Im}\int_{\mathbb{R}}\varphi\partial_tu\,\overline {\partial_xu} dx+\mbox{Im}\int_{\mathbb{R}}\varphi' \partial_tu\,\overline udx\\
=&2\mbox{Im}\int_{\mathbb{R}}\varphi(i\partial_x^2u-|u|^{2\sigma}\partial_xu)\overline {\partial_xu} dx+\mbox{Im}\int_{\mathbb{R}}\varphi' (i\partial_x^2u-|u|^{2\sigma}\partial_xu)\overline udx\\
=&-2\int_{\mathbb{R}}\varphi'|\partial_x u|^2 dx+\frac 12\int_{\mathbb{R}}\varphi'''|u|^2 dx+\mbox{Im}\int_{\mathbb{R}}\varphi'|u|^{2\sigma}u\,\overline{\partial_xu}.
\end{align*}
This proves the lemma.
\end{proof}
Now we define $\varphi_R\in C^\infty(\mathbb{R})$ satisfying
\begin{align*}
\varphi_R(x)=
\left\{ \aligned
&x,\quad |x|\leq R,\\
&2R,\quad |x|\geq 2R,
\endaligned
\right.
\end{align*}
and $0\leq |\varphi_R'|\leq 1$ for any $x\in\R$.
Moreover, we denote
\begin{align*}
I_1(t)=&\int_{\mathbb{R}}\varphi_R(x-y(t))|u|^2 dx,\\
I_2(t)=&\int_{\mathbb{R}}\varphi_R(x-y(t))\mbox{Im}(u\,\overline{\partial_xu}) dx.
\end{align*}
To prove the main theorem, we define the key functional $I(t)$ as
\begin{align*}
I(t)=-\sqrt\omega I_1(t)+I_2(t)+\widetilde C_{\omega,c}\lambda,
\end{align*}
where $\widetilde C_{\omega,c}=2C_{\omega,c}(M(\phi_{\omega,c})+P(\phi_{\omega,c}))$.

Hence, by Lemma \ref{V} we can obtain the following localized virial estimates.
\begin{lem}\label{VE}
Assume that
\begin{align}\label{*}
u=e^{i\theta(t)}\big(\phi_{\omega+\lambda(t)\mu, c+\lambda(t)\nu}+\varepsilon(t)\big)(x-y(t)),
\end{align}
with the parameters obey the estimates in Lemma \ref{doty}.
Then the following estimates hold:
\begin{align*}
I_1'(t)=&-2C_{\omega.c}\dot\lambda M(\phi_{\omega,c})-2c[M(u_0)-M(\phi_{\omega, c})]-4[P(u_0)-P(\phi_{\omega, c})]\\
&-\frac{\sigma-1}{\sigma(2-\sigma)\omega}\big\langle S''_{\omega, c}(\phi_{\omega, c})(\sqrt\omega\phi_{\omega, c}-i\partial_x\phi_{\omega, c}), \varepsilon\big\rangle
+\frac{1}{2(\sigma+1)}\lambda^2\frac{d^2}{d\lambda^2}\|\phi_{\omega+\lambda\mu, c+\lambda\nu}\|_{L^{2\sigma+2}}^{2\sigma+2}\Big|_{\lambda=0}\\
&+O(\lambda\|\varepsilon\|_{\H1}+\|\varepsilon\|_{\H1}^2+\frac 1R)+o(\lambda^2),
\end{align*}
and
\begin{align*}
I_2'(t)=&-2C_{\omega.c}\dot\lambda P(\phi_{\omega,c})-2[P(u_0)-P(\phi_{\omega, c})]-4[E(u_0)-E(\phi_{\omega, c})]\\
&-\frac{\sigma-1}{\sigma(2-\sigma)\sqrt\omega}\big\langle S''_{\omega, c}(\phi_{\omega, c})(\sqrt\omega\phi_{\omega, c}-i\partial_x\phi_{\omega, c}), \varepsilon\big\rangle
+\frac{\sigma-1}{2(\sigma+1)}\lambda^2\frac{d^2}{d\lambda^2}J(\phi_{\omega+\lambda\mu, c+\lambda\nu})\Big|_{\lambda=0}\\
&+O(\lambda\|\varepsilon\|_{\H1}+\|\varepsilon\|_{\H1}^2+\frac 1R)+o(\lambda^2).
\end{align*}
\end{lem}
\begin{proof}
By Lemma \ref{V} and some direct computation, we have the following formulas.
\begin{align*}
I_1'(t)=&-\dot y\int_{\mathbb{R}}\varphi_R'(x-y(t))|u|^2-2\mbox{Im}\int_{\mathbb{R}}\varphi_R'(x-y(t))u\,\overline{\partial_x u} dx
+\frac{1}{\sigma+1}\int_{\mathbb{R}}\varphi_R'(x-y(t))|u|^{2\sigma+2} dx,\\
I_2'(t)=&-\dot y\int_{\mathbb{R}}\varphi_R'(x-y(t))\mbox{Im}(u\,\overline{\partial_xu})dx
-\int_{\mathbb{R}}\varphi_R'(x-y(t))[2|\partial_xu|^2-\mbox{Im}(|u|^{2\sigma}u\,\overline{\partial_xu})]dx\\
&+\frac 12\int_{\mathbb{R}}\varphi_R'''(x-y(t))|u|^2 dx.
\end{align*}
First, with the definitions of $M$ and $P$ yields
\begin{align}\label{I'1}
I_1'(t)
%=&-\dot y\int_{\mathbb{R}}\varphi_R'(x-y(t))|u|^2+2\mbox{Re}\int_{\mathbb{R}}u\,\overline{\partial_t u} dx\\
%=&-\dot y\int_{\mathbb{R}}\varphi_R'(x-y(t))|u|^2+2\mbox{Re}\int_{\mathbb{R}}\varphi_R(x-y(t))u\,(-i\,\overline{\partial_x^2 u}-|u|^{2\sigma}\overline{\partial_x u}) dx\\
%=&-\dot y\int_{\mathbb{R}}\varphi_R'(x-y(t))|u|^2-2\mbox{Im}\int_{\mathbb{R}}\varphi_R'(x-y(t))u\,\overline{\partial_x u} dx
%+\frac{1}{\sigma+1}\int_{\mathbb{R}}\varphi_R'(x-y(t))|u|^{2\sigma+2} dx\\
=&-2\dot y M(u)-2\dot y\int_{\mathbb{R}}\big[\varphi_R'(x-y(t))-1\big]|u|^2 dx
-4P(u)-2\mbox{Im}\int_{\mathbb{R}}\big[\varphi_R'(x-y(t))-1\big]u\,\overline{\partial_x u}\, dx\nonumber\\
&+\frac{1}{\sigma+1}\|u\|_{L^{2\sigma+2}}^{2\sigma+2}+\frac{1}{\sigma+1}\int_{\mathbb{R}}\big[\varphi_R'(x-y(t))-1\big]|u|^{2\sigma+2} dx\nonumber\\
=&-2(\dot y-c-\lambda\nu)M(u)-2(c+\lambda\nu)M(u)-4P(u)+\frac{1}{\sigma+1}\|u\|_{L^{2\sigma+2}}^{2\sigma+2}\nonumber\\
&+O\Big(\int_{\mathbb{R}}\big[\varphi_R'(x-y(t))-1\big]\big( \dot y|u|^2+\mbox{Im}(u\,\overline{\partial_x u})+|u|^{2\sigma+2} \big) dx \Big).
\end{align}
In fact, supp[$\varphi_R'(x-y(t))-1$] $\subset$ $\{x:|x-y(t)\geq R\}$, $0\leq |\varphi_R'|\leq 1$. From Lemma \ref{doty}, we know that $|\dot y|\lesssim1$. Then, after using \eqref{*} and the exponential decaying of $\phi_{\omega+\lambda\mu, c+\lambda\nu}$, we have
\begin{align}\label{O}
\int_{\mathbb{R}}\big[\varphi_R'(x-y(t))&-1\big]\big( \dot y|u|^2+\mbox{Im}(u\,\overline{\partial_x u})+|u|^{2\sigma+2} \big) dx\nonumber\\
&\lesssim\int_{|x-y(t)|\geq R}\Big[|\phi_{\omega+\lambda\mu, c+\lambda\nu}|^2+|\partial_x^2\phi_{\omega+\lambda\mu, c+\lambda\nu}|+|\varepsilon|^2+|\partial_x\varepsilon|^2+|\varepsilon|^{2\sigma+2}\Big]\nonumber\\
&=O\big(\|\varepsilon\|_{\H1}^2+\frac 1R\big).
\end{align}
Mergering \eqref{est-derivative} and \eqref{O} into \eqref{I'1}, we obtain
\begin{align*}
I_1'(t)=&-2\big[C_{\omega,c}\dot\lambda+\big\langle S''_{\omega, c}(\phi_{\omega, c})\frac{-a_0\phi_{\omega, c}+i\partial_x\phi_{\omega, c}}{2(a_0^2-\omega)M(\phi_{\omega, c})}, \varepsilon\big\rangle+O(\lambda\|\varepsilon\|_{\H1}+\|\varepsilon\|^2_{\H1})\big]M(u)\\
&-2(c+\lambda\nu)M(u)-4P(u)+\frac{1}{\sigma+1}\|u\|_{L^{2\sigma+2}}^{2\sigma+2}+O\big(\|\varepsilon\|_{\H1}^2+\frac 1R\big).
\end{align*}
Now, using \eqref{M'P'}, \eqref{T1} and \eqref{*}, we get
\begin{align}\label{ME}
M(u)
%=&M(\phi_{\omega+\lambda\mu, c+\lambda\nu})+\langle M'(\phi_{\omega+\lambda\mu, c+\lambda\nu}), \varepsilon\rangle+O(\|\varepsilon\|_{\H1}^2)\\
%=&M(\phi_{\omega, c})+\lambda\langle M'(\phi_{\omega, c}), \psi\rangle+O(\lambda^2)
%+\langle M'(\phi_{\omega, c}), \varepsilon\rangle\\
%&+O(\lambda\|\varepsilon\|_{\H1})
%+O(\|\varepsilon\|_{\H1}^2)\\
=&M(\phi_{\omega, c})+\langle M'(\phi_{\omega, c}), \varepsilon\rangle+O(\lambda\|\varepsilon\|_{\H1}+\|\varepsilon\|_{\H1}^2),	
\end{align}
and
\begin{align}\label{PE}
P(u)
%=&P(\phi_{\omega+\lambda\mu, c+\lambda\nu})+\langle P'(\phi_{\omega+\lambda\mu, c+\lambda\nu}), \varepsilon\rangle+O(\|\varepsilon\|_{\H1}^2)\\
%=&P(\phi_{\omega, c})+\lambda\langle P'(\phi_{\omega, c}), \psi\rangle+O(\lambda^2)
%+\langle P'(\phi_{\omega, c}), \varepsilon\rangle+O(\lambda\|\varepsilon\|_{\H1})
%+O(\|\varepsilon\|_{\H1}^2)\\
=&P(\phi_{\omega, c})+\langle P'(\phi_{\omega, c}), \varepsilon\rangle+O(\lambda\|\varepsilon\|_{\H1}+\|\varepsilon\|_{\H1}^2).
\end{align}
Moreover,
\begin{align}\label{CE}
\|u\|_{L^{2\sigma+2}}^{2\sigma+2}
%=&\|\phi_{\omega+\lambda\mu, c+\lambda\nu}\|_{L^{2\sigma+2}}^{2\sigma+2}
%+2(\sigma+1)\big\langle |\phi_{\omega+\lambda\mu, c+\lambda\nu}|^{2\sigma}\phi_{\omega+\lambda\mu, c+\lambda\nu}, \varepsilon\big\rangle
%+O\big(\|\varepsilon\|_{\H1}^2\big)\\
&=\|\phi_{\omega, c}\|_{L^{2\sigma+2}}^{2\sigma+2}+2(\sigma+1)\lambda\langle|\phi_{\omega, c}|^{2\sigma}\phi_{\omega, c}, \psi\rangle
+2(\sigma+1)\langle|\phi_{\omega, c}|^{2\sigma}\phi_{\omega, c}, \varepsilon\rangle \nonumber\\
&+\frac 12\lambda^2\frac{d^2}{d\lambda^2}\|\phi_{\omega+\lambda\mu, c+\lambda\nu}\|_{L^{2\sigma+2}}^{2\sigma+2}\Big|_{\lambda=0}
+O\big(\lambda\|\varepsilon\|_{\H1}+\|\varepsilon\|_{\H1}^2\big)+o(\lambda^2),
\end{align}
and
\begin{align}\label{JE}
J(u)
%=&J(\phi_{\omega+\lambda\mu, c+\lambda\nu})+\langle J'(\phi_{\omega+\lambda\mu, c+\lambda\nu}), \varepsilon\rangle+O(\|\varepsilon\|_{\H1}^2)\\
=& J(\phi_{\omega, c})+\lambda\langle J'(\phi_{\omega, c}), \psi\rangle
+\frac 12 \lambda^2\frac{d^2}{d\lambda^2}J(\phi_{\omega+\lambda\mu, c+\lambda\nu})\big|_{\lambda=0}\nonumber\\
&+\frac{\sigma-1}{\sigma+1}\langle J'(\phi_{\omega, c}), \epsilon\rangle
+O\big(\lambda\|\varepsilon\|_{\H1}+\|\varepsilon\|_{\H1}^2\big)+o(\lambda^2).
\end{align}
Combinig \eqref{ME} and \eqref{CE} yields
\begin{align*}
I_1'(t)
%=&-2C_{\omega,c}\dot\lambda M(\phi_{\omega, c})-\big\langle S''_{\omega, c}(\phi_{\omega, c})\frac{-a_0\phi_{\omega, c}+i\partial_x\phi_{\omega, c}}{a_0^2-\omega}, \varepsilon\big\rangle\\
%&-2cM(u)-2\lambda\nu M(\phi_{\omega,c})-4P(u)
%+\frac{1}{\sigma+1}\|\phi_{\omega, c}\|_{L^{2\sigma+2}}^{2\sigma+2}\\
%&+2\lambda\langle|\phi_{\omega, c}|^{2\sigma}\phi_{\omega, c}, \psi\rangle
%+2\langle|\phi_{\omega, c}|^{2\sigma}\phi_{\omega, c}, \varepsilon\rangle
%+\frac{1}{2(\sigma+1)}\lambda^2\frac{d^2}{d\lambda^2}\|\phi_{\omega+\lambda\mu, c+\lambda\nu}\|_{L^{2\sigma+2}}^{2\sigma+2}\Big|_{\lambda=0}\\
%&+O\big(\lambda\|\varepsilon\|_{\H1}+\|\varepsilon\|_{\H1}^2+\frac 1R\big)+o(\lambda^2)\\
=&-2C_{\omega,c}\dot\lambda M(\phi_{\omega, c})-2cM(u)-2\lambda\nu M(\phi_{\omega,c})-4P(u)+\frac{1}{\sigma+1}\|\phi_{\omega, c}\|_{L^{2\sigma+2}}^{2\sigma+2}\\
&+\big\langle -S''_{\omega, c}(\phi_{\omega, c})\frac{-a_0\phi_{\omega, c}+i\partial_x\phi_{\omega, c}}{a_0^2-\omega}+2|\phi_{\omega, c}|^{2\sigma}\phi_{\omega, c}, \varepsilon\big\rangle
+2\lambda\langle|\phi_{\omega, c}|^{2\sigma}\phi_{\omega, c}, \psi\rangle\\
&+\frac{1}{2(\sigma+1)}\lambda^2\frac{d^2}{d\lambda^2}\|\phi_{\omega+\lambda\mu, c+\lambda\nu}\|_{L^{2\sigma+2}}^{2\sigma+2}\Big|_{\lambda=0}
+O\big(\lambda\|\varepsilon\|_{\H1}+\|\varepsilon\|_{\H1}^2+\frac 1R\big)+o(\lambda^2).
\end{align*}
From the conservation laws and \eqref{xphi}, we get
\begin{align*}
I_1'(t)=&-2C_{\omega,c}\dot\lambda M(\phi_{\omega, c})-2cM(u_0)-4P(u_0)-2\lambda\nu M(\phi_{\omega, c})+2cM(\phi_{\omega,c})+4P(\phi_{\omega,c})\\
&+\big\langle -S''_{\omega, c}(\phi_{\omega, c})\frac{-a_0\phi_{\omega, c}+i\partial_x\phi_{\omega, c}}{a_0^2-\omega}+2|\phi_{\omega, c}|^{2\sigma}\phi_{\omega, c}, \varepsilon\big\rangle
+2\lambda\langle|\phi_{\omega, c}|^{2\sigma}\phi_{\omega, c}, \psi\rangle\\
&+\frac{1}{2(\sigma+1)}\lambda^2\frac{d^2}{d\lambda^2}\|\phi_{\omega+\lambda\mu, c+\lambda\nu}\|_{L^{2\sigma+2}}^{2\sigma+2}\Big|_{\lambda=0}
+O\big(\lambda\|\varepsilon\|_{\H1}+\|\varepsilon\|_{\H1}^2+\frac 1R\big)+o(\lambda^2).
\end{align*}
Observe  that
\begin{align*}
\langle|\phi_{\omega, c}|^{2\sigma}\phi_{\omega, c}, \psi\rangle
=\frac{1}{2(\sigma+1)}\partial_\lambda\Big(\|\phi_{\omega+\lambda\mu, c+\lambda\nu}\|_{L^{2\sigma+2}}^{2\sigma+2}\Big)\Big|_{\lambda=0}.
\end{align*}
Moreover, using the equality \eqref{xphi} again yields
\begin{align*}
\langle|\phi_{\omega, c}|^{2\sigma}\phi_{\omega, c}, \psi\rangle
=&\frac{1}{2(\sigma+1)}\partial_\lambda\Big(4(\sigma+1)\big(\frac {c+\lambda\nu}2 M(\phi_{\omega+\lambda\mu, c+\lambda\nu})+P(\phi_{\omega+\lambda\mu, c+\lambda\nu})\big)\Big)\Big|_{\lambda=0}.
\end{align*}
Using Lemma \ref{critical} again, we obtain that
\begin{align*}
\langle|\phi_{\omega, c}|^{2\sigma}\phi_{\omega, c}, \psi\rangle
=\nu M(\phi_{\omega,c}).
\end{align*}
By Lemma \ref{S''f}, we have
\begin{align*}
-S''_{\omega, c}(\phi_{\omega, c})\frac{-a_0\phi_{\omega, c}+i\partial_x\phi_{\omega, c}}{a_0^2-\omega}
+2|\phi_{\omega, c}|^{2\sigma}\phi_{\omega, c}
%=S''_{\omega, c}(\phi_{\omega, c})\Big(\frac{-a_0\phi_{\omega, c}-i\partial_x\phi_{\omega, c}}{a_0^2-\omega}
%-\frac{1}{\sigma\omega}i\partial_x\phi_{\omega, c}\Big)
=-\frac{\sigma-1}{\sigma(2-\sigma)\omega}S''_{\omega, c}(\phi_{\omega, c})\big(\sqrt\omega\phi_{\omega, c}-i\partial_x\phi_{\omega, c}\big).
\end{align*}
Finally, we collect the above equalities and obtain
\begin{align*}
I_1'(t)=&-2C_{\omega,c}\dot\lambda M(\phi_{\omega, c})-2c[M(u_0)-M(\phi_{\omega, c})]-4[P(u_0)-P(\phi_{\omega, c})]\\
&-\frac{\sigma-1}{\sigma(2-\sigma)\omega}\big\langle S''_{\omega, c}(\phi_{\omega, c})\big(\sqrt\omega\phi_{\omega, c}-i\partial_x\phi_{\omega, c}\big), \varepsilon\big\rangle
+\frac{1}{2(\sigma+1)}\lambda^2\frac{d^2}{d\lambda^2}\|\phi_{\omega+\lambda\mu, c+\lambda\nu}\|_{L^{2\sigma+2}}^{2\sigma+2}\Big|_{\lambda=0}\\
&+O\big(\lambda\|\varepsilon\|_{\H1}+\|\varepsilon\|_{\H1}^2+\frac 1R\big)+o(\lambda^2).
\end{align*}
Similarly, from the definitions of $P$, $E$ and $J$, we have
\begin{align*}
I_2'(t)
=&-2\dot yP(u)-4E(u)+\frac{\sigma-1}{\sigma+1}J(u)\\
&-\dot y\int_{\mathbb{R}}[\varphi_R'(x-y(t))-1][\mbox{Im}(u\,\overline{\partial_xu})-2|\partial_xu|^2+\mbox{Im}(|u|^{2\sigma}u\,\overline{\partial_xu})]dx\\
=&-2(\dot y-c-\lambda\nu)P(u)-2(c+\lambda\nu)P(u)-4E(u)+\frac{\sigma-1}{\sigma+1}J(u)+O\big(\|\varepsilon\|_{\H1}^2+\frac 1R\big).
\end{align*}
Using \eqref{est-derivative}, we get
\begin{align*}
I_2'(t)=&-2\big[C_{\omega,c}\dot\lambda+\big\langle S''_{\omega, c}(\phi_{\omega, c})\frac{-a_0\phi_{\omega, c}+i\partial_x\phi_{\omega, c}}{2(a_0^2-\omega)M(\phi_{\omega, c})}, \varepsilon\big\rangle+O(\lambda\|\varepsilon\|_{\H1}+\|\varepsilon\|^2_{\H1})\big]P(u)\\
&-2(c+\lambda\nu)P(u)-4E(u)+\frac{\sigma-1}{\sigma+1}J(u)+O\big(\|\varepsilon\|_{\H1}^2+\frac 1R\big).
\end{align*}
Combining \eqref{PE} and \eqref{JE}, we obtain
\begin{align*}
I_2'(t)
=&-2C_{\omega.c}\dot\lambda P(\phi_{\omega,c})-2cP(u)-2\lambda\nu P(\phi_{\omega,c})-4E(u)+\frac{\sigma-1}{\sigma+1} J(\phi_{\omega, c})\\
&-\frac{P(\phi_{\omega, c})}{M(\phi_{\omega, c})}\big\langle S''_{\omega, c}(\phi_{\omega, c})\frac{-a_0\phi_{\omega, c}+i\partial_x\phi_{\omega, c}}{a_0^2-\omega}, \varepsilon\big\rangle+\frac{\sigma-1}{\sigma+1}\langle J'(\phi_{\omega, c}), \varepsilon\rangle\\
&+\frac{\sigma-1}{\sigma+1}\lambda \langle J'(\phi_{\omega, c}), \psi\rangle
+\frac{\sigma-1}{2(\sigma+1)}\lambda^2\frac{d^2}{d\lambda^2}J(\phi_{\omega+\lambda\mu, c+\lambda\nu})\big|_{\lambda=0}\\
&+O\big(\lambda\|\varepsilon\|_{\H1}+\|\varepsilon\|_{\H1}^2+\frac 1R\big)+o(\lambda^2).
\end{align*}
Using
%we know $P(\phi_{\omega,c})=a_0M(\phi_{\omega,c})$ and $a_0=(\sigma-1)\sqrt\omega$,
\eqref{J}, \eqref{JPE}, \eqref{J'}, Lemma \ref{Apendix} and the conservation laws, we have
\begin{align*}
I_2'(t)
=&-2C_{\omega.c}\dot\lambda P(\phi_{\omega,c})-2c[P(u_0)-P(\phi_{\omega, c})]-4[E(u_0)-E(\phi_{\omega, c})]\\
&-\frac{\sigma-1}{\sigma(2-\sigma)\sqrt\omega}\langle S''_{\omega, c}(\phi_{\omega, c})(\sqrt\omega\phi_{\omega, c}-i\partial_x\phi_{\omega, c}), \varepsilon\rangle
+\frac{\sigma-1}{2(\sigma+1)}\lambda^2\frac{d^2}{d\lambda^2}J(\phi_{\omega+\lambda\mu, c+\lambda\nu})\Big|_{\lambda=0}\\
&+O(\lambda\|\varepsilon\|_{\H1}+\|\varepsilon\|_{\H1}^2+\frac 1R)+o(\lambda^2).
\end{align*}
This completes the proof of the lemma.
\end{proof}
According to Lemma \ref{VE}, we have following result.
\begin{lem}\label{lem:I'}
Under the assumptions of Lemma \ref{VE}, we have
\begin{align*}
I'(t)=A(u_0)+B(\lambda)+O\big(\lambda\|\varepsilon\|_{\H1}+\|\varepsilon\|_{\H1}^2+\frac 1R\big)+o(\lambda^2),
\end{align*}
with $A(u_0), B(\lambda)$ verifying
\begin{align*}
A(u_0)=&(2c\sqrt\omega+4\omega)\big[(M(u_0)-M(\phi_{\omega, c}))\big]+(4\sqrt\omega-2c)\big[P(u_0)-P(\phi_{\omega, c})\big]\\
&-4\big[S_{\omega,c}(u_0)-S_{\omega,c}(\phi_{\omega, c})\big],
\end{align*}
and
\begin{align*}
B(\lambda)=b_1\lambda^2,
\end{align*}
for some $b_1>0.$
\end{lem}
\begin{remark}
 The form of $I(t)$ removes the linear term of $\varepsilon$ in $I'(t)$.
\end{remark}
\begin{proof}[Proof of Lemma \ref{lem:I'}]
From the definiton of $I(t)$, we have
\begin{align*}
I'(t)=-\sqrt\omega I'_1(t)+I'_2(t)+\widetilde C_{\omega,c}\dot\lambda.
\end{align*}
By Lemma \ref{VE}, we obtain
\begin{align*}
I'(t)=&\sqrt\omega\big[2c(M(u_0)-M(\phi_{\omega, c}))+4(P(u_0)-P(\phi_{\omega, c}))\big]\\
&-2c\big[P(u_0)-P(\phi_{\omega, c})\big]-4\big[E(u_0)-E(\phi_{\omega, c})\big]\\
&+\frac{1}{2(\sigma+1)}\frac{d^2}{d\lambda^2}\Big[-\sqrt\omega\|\phi_{\omega+\lambda\mu, c+\lambda\nu}\|_{L^{2\sigma+2}}^{2\sigma+2}+(\sigma-1)J(\phi_{\omega+\lambda\mu, c+\lambda\nu})\Big]\Big|_{\lambda=0}\\
&+O\big(\lambda\|\varepsilon\|_{\H1}+\|\varepsilon\|_{\H1}^2+\frac 1R\big)+o(\lambda^2).
\end{align*}
We denote
\begin{align*}
A(u_0)=
&\sqrt\omega\big[2c(M(u_0)-M(\phi_{\omega, c}))+4(P(u_0)-P(\phi_{\omega, c}))\big]\\
&-2c\big[P(u_0)-P(\phi_{\omega, c})\big]-4\big[E(u_0)-E(\phi_{\omega, c})\big],
\end{align*}
and
\begin{align*}
B(\lambda)=&\frac{1}{2(\sigma+1)}\frac{d^2}{d\lambda^2}\Big[-\sqrt\omega\|\phi_{\omega+\lambda\mu, c+\lambda\nu}\|_{L^{2\sigma+2}}^{2\sigma+2}+(\sigma-1)J(\phi_{\omega+\lambda\mu, c+\lambda\nu})\Big]\Big|_{\lambda=0}.
\end{align*}
Then, we have
$$
I'(t)=A(u_0)+B(\lambda)+O\big(\lambda\|\varepsilon\|_{\H1}+\|\varepsilon\|_{\H1}^2+\frac 1R\big)+o(\lambda^2).
$$
By the definiton of $S_{\omega,c}$, we have
\begin{align*}
A(u_0)
=&(2c\sqrt\omega+4\omega)\big[(M(u_0)-M(\phi_{\omega, c}))\big]+(4\sqrt\omega-2c)\big[P(u_0)-P(\phi_{\omega, c})\big]\\
&-4\big[S_{\omega,c}(u_0)-S_{\omega,c}(\phi_{\omega, c})\big].
\end{align*}
Now we consider $B(\lambda)$. Observe, from \eqref{J} and \eqref{xphi}, that
\begin{align*}
-\sqrt\omega&\|\phi_{\omega+\lambda\mu, c+\lambda\nu}\|_{L^{2\sigma+2}}^{2\sigma+2}+(\sigma-1)J(\phi_{\omega+\lambda\mu, c+\lambda\nu})\nonumber\\
=&-\sqrt\omega\cdot4(\sigma+1)[\frac{c+\lambda\nu}{2}M(\phi_{\omega+\lambda\mu, c+\lambda\nu})+P(\phi_{\omega+\lambda\mu, c+\lambda\nu})]\nonumber\\
&+(\sigma-1)[4(\omega+\lambda\mu)M(\phi_{\omega+\lambda\mu, c+\lambda\nu})+2(c+\lambda\nu)P(\phi_{\omega+\lambda\mu, c+\lambda\nu})]\nonumber\\
=&[4\omega(\sigma-1)-2c\sqrt\omega(\sigma+1)]M(\phi_{\omega+\lambda\mu, c+\lambda\nu})
+[2c(\sigma-1)-4\sqrt\omega(\sigma+1)]P(\phi_{\omega+\lambda\mu, c+\lambda\nu})\nonumber\\
&+[4\mu(\sigma-1)-2\nu\sqrt\omega(\sigma+1)]\lambda M(\phi_{\omega+\lambda\mu, c+\lambda\nu})
+2\lambda\nu P(\phi_{\omega+\lambda\mu, c+\lambda\nu}).
\end{align*}
Next, using \eqref{uvM} and \eqref{uvP}, we calculate the terms above separately:
\begin{align*}
\frac{d^2}{d\lambda^2} M(\phi_{\omega+\lambda\mu, c+\lambda\nu})\Big|_{\lambda=0}
=&\mu^2\partial_{\omega \omega}M(\phi_{\omega, c})+2\mu\nu\partial_{\omega c}M(\phi_{\omega, c})+\nu^2\partial_{c c}M(\phi_{\omega, c})
=\kappa_0z_0,\\
\frac{d^2}{d\lambda^2} P(\phi_{\omega+\lambda\mu, c+\lambda\nu})\Big|_{\lambda=0}
=&\mu^2\partial_{\omega \omega}P(\phi_{\omega, c})+2\mu\nu\partial_{\omega c}P(\phi_{\omega, c})+\nu^2\partial_{c c}P(\phi_{\omega, c})
=-\kappa_0\sqrt{\omega}.
\end{align*}
Finally, together with \eqref{M'P'}, and the three estimates above, we get
\begin{align*}
\frac{1}{2(\sigma+1)}&\frac{d^2}{d\lambda^2}\Big[-\sqrt\omega\|\phi_{\omega+\lambda\mu, c+\lambda\nu}\|_{L^{2\sigma+2}}^{2\sigma+2}+(\sigma-1)J(\phi_{\omega+\lambda\mu, c+\lambda\nu})\Big]\Big|_{\lambda=0}\\
=&\frac{1}{\sigma+1}[2\omega(\sigma-1)-\sqrt\omega c(\sigma+1)]\frac{d^2}{d\lambda^2}M(\phi_{\omega+\lambda\mu, c+\lambda\nu})\Big|_{\lambda=0}\\
&+\frac{1}{\sigma+1}[c(\sigma-1)-2\sqrt\omega(\sigma+1)]\frac{d^2}{d\lambda^2} P(\phi_{\omega+\lambda\mu, c+\lambda\nu})\Big|_{\lambda=0}\\
=&\frac{1}{\sigma+1}[2\omega(\sigma-1)-\sqrt\omega c(\sigma+1)]\cdot \kappa_0z_0\\
&+\frac{1}{\sigma+1}[c(\sigma-1)-2\sqrt\omega(\sigma+1)]\cdot-\kappa_0\sqrt{\omega}\\
=&2\kappa_0\omega(1-z_0^2).
\end{align*}
Let $b_1= 2\kappa_0\omega(1-z_0^2)$, then $b_1>0$.
Hence, we obtain that
\begin{align*}
B(\lambda)=b_1\lambda^2.
\end{align*}
This concludes the proof of Lemma \ref{lem:I'}.
\end{proof}

%First, the collection of the above lemmas yields that
%\begin{align*}
%I'(t)=A(u_0)+B(\lambda),
%\end{align*}
%where
%\begin{align*}
%A(u_0)=&\sqrt\omega\big[2c(M(u_0)-M(\phi_{\omega, c}))+4(P(u_0)-P(\phi_{\omega, c}))\big]\\
%&-2c\big[P(u_0)-P(\phi_{\omega, c})\big]-4\big[E(u_0)-E(\phi_{\omega, c})\big],
%\end{align*}
%and
%\begin{align*}
%B(\lambda)=&\frac{1}{2(\sigma+1)}\frac{d^2}{d\lambda^2}\Big[-\sqrt\omega\|\phi_{\omega+\lambda\mu, c+\lambda\nu}\|_{L^{2\sigma+2}}^{2\sigma+2}+(\sigma-1)J(\phi_{\omega+\lambda\mu, c+\lambda\nu})\Big]\Big|_{\lambda=0}\\
%&+O\big(\lambda\|\varepsilon\|_{\H1}+\|\varepsilon\|_{\H1}^2+\frac 1R\big)+o(\lambda^2).
%\end{align*}
%By the definiton of $S_{\omega,c}$, we have
%\begin{align*}
%A(u_0)=&(2\sqrt\omega\, c+4\omega)\big[(M(u_0)-M(\phi_{\omega, c}))\big]+(4\sqrt\omega-2c)\big[P(u_0)-P(\phi_{\omega, c})\big]\\
%&-4\big[S_{\omega,c}(u_0)-S_{\omega,c}(\phi_{\omega, c})\big].
%\end{align*}
%
%Then, we get
%\begin{align*}
%B(\lambda)=b_1\lambda^2+O\big(\lambda\|\varepsilon\|_{\H1}+\|\varepsilon\|_{\H1}^2+\frac 1R\big)+o(\lambda^2).
%\end{align*}
%Hence, we have
%\begin{align*}
%I'(t)=&(2\sqrt\omega\, c+4\omega)\big[(M(u_0)-M(\phi_{\omega, c}))\big]+(4\sqrt\omega-2c)\big[P(u_0)-P(\phi_{\omega, c})\big]\\
%-&4\big[S_{\omega,c}(u_0)-S_{\omega,c}(\phi_{\omega, c})\big]+b_1\lambda^2+O\big(\lambda\|\varepsilon\|_{\H1}+\|\varepsilon\|_{\H1}^2+\frac 1R\big)+o(\lambda^2).
%\end{align*}

\subsection{Proof of Theorem \ref{thm:mainthm}}
Now we give the proof of Theorem \ref{thm:mainthm}. Suppose that $e^{i\omega t}\phi_{\omega,c}(x-ct)$ of \eqref{eqs:gDNLS} is stable. Choose
$$u_0=\phi_{\omega, c}+\delta_1(-a_0\phi_{\omega, c}+i\partial_x\phi_{\omega, c}),\quad \delta_1>0.$$
Here $\delta_1$ is small enough such that $u_0\in U_\delta(\phi_{\omega,c})$ which is given by Proposition \ref{Modulation}.

Let $u$ be the corresponding solution of \eqref{eqs:gDNLS} with the initial data $u_0$. Then, we can write
$$u=e^{i\theta}\big(\phi_{\omega+\lambda\mu, c+\lambda\nu}+\varepsilon\big)(x-y),$$
with $(\theta, y, \lambda)$ obtained in Proposition \ref{Modulation}, and $|\lambda|\ll 1$.
\begin{lem}\label{A(u0)}
There exists $b_2>0$, such that
\begin{align*}
A(u_0)\geq b_2\delta_1.
\end{align*}
\end{lem}
\begin{proof}
Recalling that $c=2z_0\sqrt\omega$, $a_0=(\sigma-1)\sqrt\omega$ and the choose of $u_0$, we have
\begin{align}
M(u_0)-M(\phi_{\omega, c})=&\delta_1\langle M'(\phi_{\omega, c}), -a_0\phi_{\omega, c}+i\partial_x\phi_{\omega, c}\rangle+o(\delta_1)\nonumber\\
=&\delta_1[-2a_0M(\phi_{\omega,c})+2P(\phi_{\omega,c})]+o(\delta_1)\nonumber\\
=&o(\delta_1)\label{M:u_0-phi},
\end{align}
and
\begin{align}
P(u_0)-P(\phi_{\omega, c})=&\delta_1\langle P'(\phi_{\omega, c}), -a_0\phi_{\omega, c}+i\partial_x\phi_{\omega, c}\rangle+o(\delta_1)\nonumber\\
=&2(\omega-a_0^2)M(\phi_{\omega, c})\delta_1+o(\delta_1)\nonumber\\
=&2\omega\sigma(2-\sigma) M(\phi_{\omega, c})\delta_1+o(\delta_1)\label{P:u_0-phi}.
\end{align}
Moreover, using $S'_{\omega,c}(\phi_{\omega,c})=0$, we get
\begin{align}
S_{\omega,c}(u_0)-S_{\omega,c}(\phi_{\omega, c})=&\delta_1\langle S'_{\omega,c}(\phi_{\omega,c}), -a_0\phi_{\omega, c}+i\partial_x\phi_{\omega, c}\rangle+o(\delta_1)\nonumber\\
=&o(\delta_1)\label{S:u_0-phi}.
\end{align}
Now, we collect the above computations and obtain
\begin{align*}
A(u_0)=&(2c\sqrt\omega+4\omega)\cdot o(\delta_1)+(4\sqrt\omega-2c)\cdot 2\omega\sigma(2-\sigma)M(\phi_{\omega, c})\delta_1+o(\delta_1)-4\cdot o(\delta_1)\\
=&8\omega\sqrt\omega\,\sigma(2-\sigma)M(\phi_{\omega,c})\delta_1+o(\delta_1)\\
\geq&b_2\delta_1,
\end{align*}	
where choose $b_2=4\omega\sqrt\omega\,\sigma(2-\sigma)M(\phi_{\omega,c})>0$. This proves the lemma.
\end{proof}

We further give the estimate on $\|\varepsilon\|_{\H1}^2$.
\begin{lem}\label{eee}
Let $\varepsilon$ be defined in \eqref{modulation-u}, there exists $b_3>0$,  then
$$\|\varepsilon\|_{\H1}^2\leq b_3\lambda\delta_1.$$
\end{lem}
\begin{proof} Without loss of generality, we may assume that $\nu>0$.
From the conservation laws, we have
\begin{align*}
S_{\omega+\lambda\mu, c+\lambda\nu}(u_0)=&S_{\omega+\lambda\mu, c+\lambda\nu}(u)\\
=&S_{\omega+\lambda\mu, c+\lambda\nu}(\phi_{\omega+\lambda\mu, c+\lambda\nu})
+\frac 12\big\langle S''_{\omega+\lambda\mu, c+\lambda\nu}(\phi_{\omega+\lambda\mu, c+\lambda\nu})\varepsilon, \varepsilon\big\rangle
+o(\|\varepsilon\|_{\H1}^2).
\end{align*}
Combining \eqref{S'wc} and Lemma \ref{critical} yields
\begin{align*}
S_{\omega+\lambda\mu, c+\lambda\nu}(\phi_{\omega+\lambda\mu, c+\lambda\nu})
=&S_{\omega+\lambda\mu, c+\lambda\nu}(\phi_{\omega, c})
+\lambda\big\langle S'_{\omega+\lambda\mu, c+\lambda\nu}(\phi_{\omega, c}), \psi\big\rangle\\
&+\frac 12\lambda^2\big\langle S''_{\omega+\lambda\mu, c+\lambda\nu}(\phi_{\omega, c})\psi, \psi\big\rangle+o(\lambda^2)\\
=&S_{\omega+\lambda\mu, c+\lambda\nu}(\phi_{\omega, c})
+\lambda\langle S'_{\omega, c}(\phi_{\omega, c}), \psi\rangle+\lambda\mu\langle M'(\phi_{\omega,c}), \psi\rangle\\
&+\lambda\nu\langle P'(\phi_{\omega,c}), \psi\rangle
+\frac 12\lambda^2\big\langle S''_{\omega, c}(\phi_{\omega, c})\psi, \psi\big\rangle+o(\lambda^2)\\
=&S_{\omega+\lambda\mu, c+\lambda\nu}(\phi_{\omega, c})+o(\lambda^2).
\end{align*}
Then, we have
\begin{align*}
S_{\omega+\lambda\mu, c+\lambda\nu}(u_0)
=S_{\omega+\lambda\mu, c+\lambda\nu}(\phi_{\omega, c})
+\frac 12\big\langle S''_{\omega+\lambda\mu, c+\lambda\nu}(\phi_{\omega+\lambda\mu, c+\lambda\nu})\varepsilon, \varepsilon\big\rangle
+o(\lambda^2+\|\varepsilon\|_{\H1}^2).
\end{align*}
Together with \eqref{M:u_0-phi}, \eqref{P:u_0-phi} and \eqref{S:u_0-phi}, we have
\begin{align}
S_{\omega+\lambda\mu, c+\lambda\nu}(u_0)-&S_{\omega+\lambda\mu, c+\lambda\nu}(\phi_{\omega, c})\nonumber\\
=&S_{\omega,c}(u_0)-S_{\omega,c}(\phi_{\omega,c})+\lambda\mu(M(u_0)-M(\phi_{\omega,c}))+\lambda\nu(P(u_0)-P(\phi_{\omega,c}))\nonumber\\
=&2\nu\omega\sigma(2-\sigma)M(\phi_{\omega,c})\lambda\delta_1+o(\delta_1)\label{SSS}.
\end{align}
Therefore, by \eqref{coercivity}, \eqref{dotl} and \eqref{SSS}, there exists $C>0$, such that
\begin{align*}
\|\varepsilon\|_{\H1}^2\le &
C\big\langle S''_{\omega+\lambda\mu, c+\lambda\nu}(\phi_{\omega+\lambda\mu, c+\lambda\nu})\varepsilon, \varepsilon\big\rangle\\
=&C\big[S_{\omega+\lambda\mu, c+\lambda\nu}(u_0)-S_{\omega+\lambda\mu, c+\lambda\nu}(\phi_{\omega, c})\big]+o(\lambda^2+\|\varepsilon\|_{\H1}^2)\\
%=&\lambda\langle Q'_{\mu, \nu}(\phi_{\omega, c}), u_0-\phi_{\omega, c}\rangle+o(\lambda^2+\|\varepsilon\|_{\H1}^2)\\
=&2C\nu\omega\sigma(2-\sigma)M(\phi_{\omega,c})\lambda\delta_1+o(\delta_1)+o(\lambda^2+\|\varepsilon\|_{\H1}^2)\\
\leq& 2b_3\lambda\delta_1+o(\|\varepsilon\|_{\H1}^2),
\end{align*}
where $b_3=2C\nu\omega\sigma(2-\sigma)M(\phi_{\omega, c})>0$.
Then we obtain
$$\|\varepsilon\|_{\H1}^2\leq b_3\lambda\delta_1.$$
This completes the proof.
\end{proof}

\begin{proof}[Proof of Theorem \ref{thm:mainthm}]
On the one hand, we note that from the definition of $I(t)$, we have the time uniform boundedness of $I(t)$. That is, if $|\lambda|\lesssim 1$, then 
\begin{align}\label{IB}
\sup_{t\in\R} I(t)\lesssim R(\|\phi_{\omega,c}\|_{\H1}^2+1).
\end{align}
On the other hand, using \eqref{dotl} and Lemmas \ref{lem:I'}, \ref{A(u0)}, we get
\begin{align*}
I'(t)=&A(u_0)+B(\lambda)+O\big(\lambda\|\varepsilon\|_{\H1}+\|\varepsilon\|_{\H1}^2+\frac 1R\big)+o(\lambda^2)\\
\geq&b_2\delta_1+b_1\lambda^2+O\big(\|\varepsilon\|_{\H1}^2+\frac 1R\big)+o(\lambda^2)\\
\geq&\frac 12 b_2\delta_1+\frac 12 b_1\lambda^2+O(\|\varepsilon\|_{\H1}^2),
\end{align*}
where choosing $R\geq 10(b_2\delta_1)^{-1}$.

Moreover, combining Lemme \ref{eee} yields
\begin{align*}
I'(t)\geq\frac 14 b_2\delta_1+\frac 12b_1\lambda^2>0,
\end{align*}
when $|\lambda|\ll1$.

This implies that $I(t)\rightarrow +\infty$ when $t\rightarrow +\infty$, which is contradicted with \eqref{IB}. Hence we prove the instablility of the solitary wave solutions $e^{i\omega t}\phi_{\omega,c}(x-ct)$ of \eqref{eqs:gDNLS}. This completes the proof of Theorem \ref{thm:mainthm}.
\end{proof}

\section*{Appendix: Proof of  Lemma \ref{Apendix1} and Lemma \ref{Apendix}}
Throughout this section, let $1<\sigma<2$ and $z_{0}=z_{0}(\sigma) \in(-1,1)$ satisfy $F_\sigma(z_0)=0$.

Now we adopt some notations from \cite{LiSiSu1}. More precisely, for any $(\omega,c)\in \R^2$ satisfying $c^2<4\omega$, we denote
\begin{align*}
&\kappa=\sqrt{4\omega-c^2}>0, \quad
\tilde\kappa=2^{\frac 1\sigma-2}\sigma^{-1}(1+\sigma)^{\frac 1\sigma}\kappa^{\frac 2\sigma-2}\omega^{-\frac{1}{2\sigma}-\frac 12},\quad
f(\omega,c)=\frac{(\sigma+1)\kappa^2}{2\sqrt{\omega}}, \\
&h(x; \sigma; \omega, c)=\cosh(\sigma\kappa x)-\frac{c}{2\sqrt{\omega}}, \quad \quad \quad
\alpha_n=\int_{0}^{\infty} h^{-{\frac{1}{\sigma}}-n}dx,\ n\in\Z^+.
\end{align*}

\begin{proof}[Proof of Lemma \ref{Apendix1}]
For any $(\omega,c)\in \R^2$ satisfy $c^2<4\omega$, by \eqref{phi}, we have
$$\partial_x\phi_{\omega,c}(x)=e^{i\theta}\big[(i\frac c2-\frac{i}{2\sigma+2}\varphi_{\omega,c}^{2\sigma})\varphi_{\omega,c}+\partial_x\varphi_{\omega,c}\big].$$
Therefore, we have
\begin{align*}
P(\phi_{\omega,c})=&\frac 12 \mbox{Im}\int_{\mathbb{R}}\phi_{\omega,c}\overline{\partial_x\phi_{\omega,c}} dx\\
=&\frac 12 \mbox{Im}\int_{\mathbb{R}}-\big[i\frac c2-\frac{i}{2\sigma+2}\varphi_{\omega,c}^{2\sigma}\big]\varphi_{\omega,c}^2 dx\\
=&-\frac c4\|\varphi_{\omega,c}\|_{L^2}^2+\frac{1}{4(\sigma+1)}\|\varphi_{\omega,c}\|_{L^{2\sigma+2}}^{2\sigma+2}.
\end{align*}
Finally, we obtain
\begin{align*}
\|\phi_{\omega,c}\|_{L^{2\sigma+2}}^{2\sigma+2}=4(\sigma+1)\big[\frac c2M(\phi_{\omega,c})+P(\phi_{\omega,c})\big].
\end{align*}
According to  \cite{LiSiSu1} Appendix Lemma A.3, we have that
$$\partial_cM(\phi_{\omega,c})=\partial_\omega P(\phi_{\omega,c}), \quad   \partial_cP(\phi_{\omega,c})=\omega\partial_\omega M(\phi_{\omega,c}).$$
This completes the proof.
\end{proof}

Now, we focus on the critical case $c=2z_0\sqrt\omega$.
\begin{proof}[Proof of Lemma \ref{Apendix}]
From \cite{LiSiSu1} Lemma 4.2, $\det[d''(\omega,c)]=0$ is equivalent to
$$\big[{(\sigma-1)}\sqrt{\omega}M(\phi_{\omega,c})\big]^2=P(\phi_{\omega,c})^2.$$
When $c=2z_0\sqrt\omega$, we have $P(\phi_{\omega,c})>0$.
Indeed, if $P(\phi_{\omega,c})<0$, since $\frac{P(\phi_{\omega,c})}{M(\phi_{\omega,c})}\rightarrow +\infty$, where  $c\rightarrow-2\sqrt{\omega}$. Then there exist two solutions $c_1=c_1(\sqrt{\omega})$, $c_2=c_2(\sqrt{\omega})$, such that
$$\big|\frac{P(\phi_{\omega,c})}{M(\phi_{\omega,c})}\big|=(\sigma-1)\sqrt{\omega}.$$
This contradicts the fact that  $z_0$ is the unique solution of $\det[d''(\omega,c)]=0$. Hence,
$$P(\phi_{\omega,c})>0, \quad P(\phi_{\omega,c})=(\sigma-1)\sqrt\omega M(\phi_{\omega,c}).$$

From \cite{LiSiSu1} Appendix (A.2) Lemmas A.1 and A.2, we know that $M(\phi_{\omega,c})=f^{\frac 1\sigma}\alpha_0$ and
$P(\phi_{\omega,c})
%=\frac{1}{4\sqrt{\omega}}f^{\frac{1}{\sigma}}(-2\sqrt{\omega}c\alpha_0+\kappa^2\alpha_1)
=\frac{1}{4\sqrt\omega}f^{\frac 1\sigma}(-2\omega^{\frac 12}c\alpha_0+\kappa^2\alpha_1).$
Since  $P(\phi_{\omega,c})>0$, then we have
$$\kappa^2\alpha_1>2\omega^{\frac 12}c\alpha_0.$$
Together with \cite{LiSiSu1} Appendix Lemma A.3, we obtain
\begin{align*}
\partial_\omega M(\phi_{\omega,c})
%=&\frac{1}{\sigma}f^{\frac{1-\sigma}{\sigma}}f_\omega\int_{0}^{\infty} h^{-\frac{1}{\sigma}} dx
%-\frac{1}{\sigma}f^{\frac{1}{\sigma}}\int_{0}^{\infty} h^{-\frac{\sigma+1}{\sigma}} dx\\
=&\tilde{\kappa}\omega^{-1}[-8(\sigma-1)\omega^{\frac 23}\alpha_0+c(2\omega^{\frac 1 2}c\alpha_0-\kappa^2\alpha_1)]<0,\\
\partial_\omega P(\phi_{\omega,c})=&2\tilde{\kappa}[2c\,\omega^{\frac 12}\alpha_0(\sigma-1)-2c\,\omega^{\frac 12}\alpha_0+\kappa^2\alpha_1]>0.
\end{align*}
On the one hand, by $\partial_cP(\phi_{\omega,c})=\omega\partial_\omega M(\phi_{\omega,c})$ and $ \partial_\omega P(\phi_{\omega,c})=\partial_c M(\phi_{\omega,c})$, we have
$$\frac{\mu}{\nu}=-\frac{\partial_c M(\phi_{\omega,c})}{\partial_\omega M(\phi_{\omega,c})}=-\frac{\partial_\omega P(\phi_{\omega,c})}{\partial_\omega M(\phi_{\omega,c})}>0.$$
On the other hand,
$$\frac \mu\nu=-\frac{\partial_c P(\phi_{\omega,c})}{\partial_\omega P(\phi_{\omega,c})}=-\frac{\omega\partial_\omega M(\phi_{\omega,c})}{\partial_\omega P(\phi_{\omega,c})}.$$
Hence, combining with above, we get
$$\big(\frac{\mu}{\nu}\big)^2=-\frac{\partial_\omega P(\phi_{\omega,c})}{\partial_\omega M(\phi_{\omega,c})}\cdot -\frac{\omega\partial_\omega M(\phi_{\omega,c})}{\partial_\omega P(\phi_{\omega,c})}=\omega.$$
Then we obtain
$$\frac\mu\nu=\sqrt{\omega}.$$

Differentiating $M(\phi_{\omega,c})$ and $P(\phi_{\omega,c})$ with respect to $\omega$ and $c$, we have the following relations:
\begin{align*}
\partial_{\omega \omega}M(\phi_{\omega,c})=&\frac 1\omega \big(\partial_{\omega c}P(\phi_{\omega,c})-\partial_\omega M(\phi_{\omega,c})\big),\quad \quad
\partial_{\omega c}M(\phi_{\omega,c})=\partial_{\omega \omega}P(\phi_{\omega,c}),\\
\partial_{c c}M(\phi_{\omega,c})=&\partial_{\omega c}P(\phi_{\omega,c}),\quad \quad \quad\quad\quad\quad\quad
\partial_{c c}P(\phi_{\omega,c})=\omega \partial_{\omega \omega}P(\phi_{\omega,c}).
\end{align*}
Since $\partial_\omega P(\phi_{\omega,c})=\sqrt\omega\partial_\omega M(\phi_{\omega,c})$, we obtain $(z_0^2-1)\alpha_1=(1-z_0-\sigma)\alpha_0$.  From \cite{Fu-16-DNLS} Appendix Lemma 10, we have
\begin{align*}
\partial_\omega M(\phi_{\omega,c})
=8\sqrt\omega\tilde\kappa\alpha_0(z_0^2-\sigma+1)+8\sqrt\omega\tilde\kappa\alpha_1z_0(z_0^2-1)
%=8\sqrt\omega\tilde\kappa\alpha_0\big[(z_0^2-\sigma+1)+z_0(1-z_0-\sigma)\big]
=8\sqrt\omega\tilde\kappa\alpha_0(1-\sigma)(1+z_0),
\end{align*}
$$2\sqrt\omega\partial_{\omega \omega}P(\phi_{\omega,c})+2\partial_{\omega c}P(\phi_{\omega,c})-\frac 12\partial_\omega M(\phi_{\omega,c})=-4\sqrt\omega\tilde\kappa \alpha_0(\sigma-1)(1-z_0),$$
and
\begin{align*}
\mu^2\partial_{\omega \omega}P(\phi_{\omega,c})+2\mu\nu\partial_{\omega c}P(\phi_{\omega,c})+\nu^2\partial_{c c}P(\phi_{\omega,c})
%=\nu^2\big[2\omega\partial_{\omega \omega}P+2\sqrt{\omega} \partial_{\omega c}P\big]
%=\nu^2\sqrt\omega\big[2\sqrt\omega\partial_{\omega \omega}P+2\partial_{\omega c}P \big]
=-8\nu^2\omega\tilde{\kappa}_\omega\alpha_0(\sigma-1).
\end{align*}
Moreover, we have
\begin{align*}
\mu^2\partial_{\omega \omega}&M(\phi_{\omega,c})+2\mu\nu\partial_{\omega c}M(\phi_{\omega,c})+\nu^2\partial_{c c}M(\phi_{\omega,c})\\
=&\nu^2\big[\omega \partial_{\omega \omega}M(\phi_{\omega,c})+2\sqrt{\omega}\partial_{\omega c}M(\phi_{\omega,c})+\partial_{c c}M(\phi_{\omega,c})\big]\\
=&\nu^2\big[-\partial_\omega M(\phi_{\omega,c})+ 2\partial_{\omega c}P(\phi_{\omega,c})+2\sqrt{\omega}\partial_{\omega  \omega}P(\phi_{\omega,c})\big]\\
=&\nu^2[2\sqrt\omega\partial_{\omega \omega}P(\phi_{\omega,c})+2\partial_{\omega c}P(\phi_{\omega,c})-\frac 12\partial_\omega M(\phi_{\omega,c})-\frac 12\partial_\omega M(\phi_{\omega,c})]\\
=&\nu^2[-4\sqrt\omega\tilde\kappa \alpha_0(\sigma-1)(1-z_0)+4\sqrt\omega\tilde{\kappa}_\omega\alpha_0(1-\sigma)(z_0+1)]\\
=&8\nu^2\sqrt\omega\tilde\kappa\alpha_0(\sigma-1)z_0.
\end{align*}
Take $\kappa_0=8\nu^2\sqrt\omega\tilde\kappa_\omega\alpha_0(\sigma-1)$, then $\kappa_0>0$. 
This concludes the proof of Lemma \ref{Apendix}.
\end{proof}

\end{document}